\documentclass[a4paper,10pt,intlimits,oneside,reqno]{amsart}
\usepackage{enumerate}
\usepackage{amsfonts,amsmath}

\usepackage{latexsym,amssymb}
\newtheorem{theorem}{Theorem}[section]

\theoremstyle{corollary}

\theoremstyle{proposition}
\newtheorem{proposition}[theorem]{Proposition}

\theoremstyle{definition}
\newtheorem{definition}[theorem]{Definition}

\theoremstyle{remark}

\numberwithin{equation}{section}

\newcommand{\comment}[1]{}

\begin{document}

\title [Hausdorff operators  in two weighted Herz-type Hardy spaces]{Two-weighted inequalities for Hausdorff operators
\\
in  Herz-type Hardy spaces}

\thanks{This paper is funded by Vietnam National Foundation for Science and Technology Development (NAFOSTED)}
\author{Nguyen Minh Chuong}
\address{Institute of mathematics, Vietnamese Academy of Science and Technology,
Hanoi, Vietnam}
\email{nmchuong@math.ac.vn}

\author{Dao Van Duong}
\address{School of Mathematics, Mientrung University of Civil Engineering, Phu Yen, Vietnam}
\email{daovanduong@muce.edu.vn}

\author{Kieu Huu Dung}
\address{School of Mathematics, University of Transport and Communications, Ha Noi, Vietnam}
\email{khdung@utc2.edu.vn}
\keywords{Hausdorff operator, Herz space, Herz-type Hardy space, $A_p$ weight, atom}
\subjclass[2010]{42B25, 42B99, 26D15}
\begin{abstract}
In this paper, we prove the  boundedness of matrix Hausdorff operators and rough Hausdorff operators in the two weighted Herz-type Hardy spaces  associated with both power weights and  Muckenhoupt weights. By applying the fact that the standard infinite atomic decomposition norm on two weighted Herz-type Hardy spaces is equivalent to the finite atomic norm on some dense subspaces of them, we generalize some previous known results due to Chen et al. \cite{CFL2012} and Ruan, Fan \cite{RF2016}.
\end{abstract}

\maketitle

\section{Introduction}
It is well known that the Hausdorff operator is one of important operators in harmonic analysis, and it is used to solve certain classical problems in analysis, especially it is closely related to the summability of the classical
Fourier series  (see, for instance, \cite{Andersen2}, \cite{BM}, \cite{CFL2012}, \cite{Liflyand1}, \cite{Liflyand2} and the references therein).  Let $\Phi$ be a locally integrable function on $\mathbb R^n$. The matrix Hausdorff operator $H_{\Phi,A}$ associated to the kernel function $\Phi$ is then defined in terms of the integral form as follows
\begin{equation}\label{Hausdorff1}
H_{\Phi, A}(f)(x)=\int\limits_{\mathbb R^n}{\frac{\Phi(y)}{|y|^n}f(A(y) x)dy},\,x\in\mathbb R^n,
\end{equation}
where $A(y)$ is an $n\times n$ invertible matrix for almost everywhere $y$ in the support of $\Phi$. It is worth pointing out that if the kernel function $\Phi$ is chosen appropriately, then the Hausdorff operator reduces to many classcial operators in analysis such as the Hardy operator, the Ces\`{a}ro operator, the Riemann-Liouville fractional integral operator and the Hardy-Littlewood average operator (see, e.g., \cite{Andersen1}, \cite{Chuong2016},  \cite{CDH2016}, \cite{Christ}, \cite{Dzherbashyan} \cite{FGLY2015}, \cite{Miyachi}, \cite{Moricz2005}, \cite{Xiao} and references therein).
\vskip 5pt
In  2012,  Chen,  Fan and  Li \cite{CFL2012} introduced the rough Hausdorff operator on $\mathbb R^n$.  More precisely, 
let $\Phi$ be a locally integrable and radial  function on $\mathbb R^n$ and $\Omega:S_{n-1}\longrightarrow\mathbb C $ be measurable functions such that $\Omega(y)\neq 0$ for almost everywhere $y$ in $S_{n-1}$. The rough Hausdorff operator ${\mathcal {H}}_{\Phi,\Omega}$  is then defined by
\begin{equation}\label{RoughHausdorf}
{\mathcal {H}}_{\Phi,\Omega}(f)(x)=\int_{\mathbb R^n}\dfrac{\Phi(x|y|^{-1})}{|y|^n}\Omega(y')f(y)dy,\;\;x\in\mathbb R^n,
\end{equation}
where $y'=\frac{y}{|y|}$.  Remark that using polar coordinates, we can rewrite 
\begin{equation}\label{RH1}
{\mathcal H}_{\Phi,\Omega}(f)(x)=\int_{0}^{\infty}\int_{S_{n-1}} \dfrac{\Phi(t)}{t}\Omega(y')f(t^{-1}|x|y')dy'dt,\;\;x\in\mathbb R^n.
\end{equation}
By choosing $\Omega=1$, we denote ${\mathcal H}_{\Phi}:={\mathcal H}_{\Phi,\Omega}$.
\vskip 5pt
Moreover,  it is interesting that Chuong, Duong and Dung \cite{CDD2017} introduced a general class of multilinear Hausdorff operators  defined by
\begin{equation}\label{mulHausdorff}
{H_{\Phi ,\vec{A} }}(\vec{f})(x) = \int\limits_{{\mathbb R^n}} {\frac{{\Phi (y)}}{{{{\left| y \right|}^n}}}} \prod\limits_{i = 1}^m {{f_i}} ({A_i}(y)x)dy,\,x\in\mathbb R^n,
\end{equation}
for $\vec{f}=\left(f_1, ..., f_m\right)$ and $\vec{A}=\left(A_1, ..., A_m\right)$. The authors gave necessary and sufficient conditions for the boundedness of ${H_{\Phi ,\vec{A} }}$ on the weighted Lebesgue, Herz, central Morrey and Morrey-Herz type spaces with variable exponent.
\vskip 5pt
It is well known that in recent years, the theory of  Hausdorff type operators  has been significantly developed into different contexts  (see \cite{Andersen2}, \cite{CFL2012}, \cite{CDH2016}, \cite{Christ}, \cite{FGLY2015}, \cite{Miyachi}, \cite{Moricz2005}, \cite{Tang}). Especially, the problem which establishes the boundedness for Hausdorff operators in the Hardy spaces is attractive to mathematicians. However, for all we know, there is no any work dealing with the study of the Hausdorff operators on the Hardy spaces $H^p(\mathbb R^n)$ for the case $n\geq 2$ and $0<p<1$.  Liflyand and Miyachi \cite{LM2009} even  showed that, in the case $n = 1$, there exists a bounded function $\Phi$ whose support is contained in $[a,b] \subset (0,\infty)$ such that the Hausdorff operator $H_{\Phi}$, which is defined by
\begin{equation}
H_{\Phi}f(x)=\int_{0}^{\infty}\frac{\Phi(y)}{y}f\left(\frac{x}{y}\right)dy, 
\end{equation}
 is not bounded on $H^p(\mathbb R)$ for any $0 < p < 1$. Thus, it is natural to find some other spaces that are the right substitutes to the Hardy spaces. Very recently, the authors of the papers \cite{CFL2012, RF2016} have showed that if the Hardy spaces are replaced by the Herz-type Hardy spaces, then the boundedness of the Hausdorff operators is solved.
\vskip 5pt
The theory of Hardy spaces associated with Herz spaces has developed in the past few years and played important roles in harmonic analysis, partial differential equation (see \cite{C1989}, \cite{CH1994}, \cite{LY1995}, \cite{LY1995(1)}, \cite{LY1996}, \cite{LY1997} for more details). These new Hardy spaces can be regarded as the local version at the origin of the classical Hardy spaces $H^p(\mathbb R^n)$ and are good substitutes for $H^p(\mathbb R^n)$ when we study the boundedness of non-translation invariant operators (see, for example, \cite{LY1995(2)}). From the results of Meyer \cite{MTW1985, MC1997}, Bownik \cite{B2005}, Yabuta \cite{Y1993},  Yang \cite{YZ2008, YZ2009}, Meda \cite{MSV2008} and Grafakos \cite{GLY2008}, the author of the paper \cite{Z2009} proved that the norms in two weighted Herz-type Hardy spaces ${H\mathop{K}\limits^{.}}^{\alpha, p}_{q}(\omega_1,\omega_2)$ can be achieved by finite central atomic decompositions in some dense subspaces of them. As an application, it is shown that if $T$ is a sublinear operator and maps all central $(\alpha, q, s,\omega_1,\omega_2)_0$-atoms into uniformly bounded elements of certain quasi-Banach space $\mathcal B$ for certain nonnegative integer $s\geq [\alpha-n(1-1/q)]$, then $T$ uniquely extends to a bounded sublinear operator from ${H\mathop{K}\limits^{.}}^{\alpha, p}_{q}(\omega_1,\omega_2)$ to $\mathcal B$.
\vskip 5pt
In this paper, by using above mentioned method which is quite different from the previous method \cite{RF2016}, we establish the sufficient conditions for  the boundedness of both matrix Hausdorff operators  ${{H}}_{\Phi,A}$ and rough Hausdorff operator ${\mathcal {H}}_{\Phi,\Omega}$ on the two-weighted homogeneous Herz-type Hardy spaces ${H\mathop{K}\limits^{.}}^{\alpha, p}_{q}(\omega_1,\omega_2)$.
\vskip 5pt
 Our paper is organized as follows. In Section \ref{2}, we present some notations and definitions of the homogeneous Herz spaces and the homogeneous Herz-type Hardy spaces associated with two weights. Our main theorems are given and proved in Section \ref{3} and Section \ref{4}.
\section{Some notations and definitions}\label{2}
Throught the whole paper, we denote by $C$ a positive geometric constant that is independent of the main parameters, but can change from line to line. We also write $a \lesssim b$ to mean that there is a positive constant $C$, independent of the main parameters, such that $a\leq Cb$.
\vskip 5pt
It is well known that the theory of $A_p$ weight was first introduced by Muckenhoupt \cite{Muckenhoupt1972} in the Euclidean spaces in order to characterise  the weighted $L^p$ boundedness of Hardy-Littlewood maximal functions. 
\begin{definition} Let $1 < p < \infty$. It is said that a weight $\omega \in A_p(\mathbb R^n)$ if there exists a constant $C$ such that for all balls $B\subset\mathbb R^n$,
$$\Big(\dfrac{1}{|B|}\int_{B}\omega(x)dx\Big)\Big(\dfrac{1}{|B|}\int_{B}\omega(x)^{-1/(p-1)}dx\Big)^{p-1}\leq C.$$
It is said that a weight $\omega\in A_1(\mathbb R^n)$ if there is a constant $C$ such that for all balls $B\subset\mathbb R^n$,
$$\dfrac{1}{|B|}\int_{B}\omega(x)dx\leq C\mathop{\rm essinf}\limits_{x\in B}\omega(x).$$
We denote $A_{\infty}(\mathbb R^n) = \bigcup\limits_{1\leq p<\infty}A_p(\mathbb R^n)$.
\end{definition}
\vskip 5pt
Remark that a close relation to $A_{\infty}(\mathbb R^n)$ is the reverse H\"{o}lder condition. If there exist $r > 1$ and a fixed constant $C$ such that
$\big(\frac{1}{|B|}\int_{B}\omega(x)^rdx\big)^{1/r}\leq \frac{C}{|B|}\int_{B}\omega(x)dx,$ for all balls $B \subset\mathbb R^n$, we then say that $\omega$ satisfies the reverse H\"{o}lder condition of order $r$ and write $\omega\in RH_r(\mathbb R^n)$.
\\
According to Theorem 19 and Corollary 21 in \cite{IMS2015}, $\omega\in A_{\infty}(\mathbb R^n)$ if and
only if there exists some $r > 1$ such that $\omega\in RH_r(\mathbb R^n)$. Moreover, if $\omega\in RH_r(\mathbb R^n)$, $r > 1$, then $\omega\in RH_{r+\varepsilon}(\mathbb R^n)$ for some $\varepsilon > 0$. We thus write $r_\omega \equiv {\rm sup}\{r > 1: \omega\in RH_r(\mathbb R^n)\}$ to denote the critical index of $\omega$ for the reverse H\"{o}lder condition. For further properties of $A_p$ weights, one may find in the book \cite{Stein1993}.
\vskip 5pt
\begin{proposition} The following statements are true:
\begin{itemize}
\item[(i)] $|x|^{\alpha}_h\in A_1(\mathbb R^n)$ if and only if $-n< \alpha\leq 0$;
\item[(ii)] $|x|^{\alpha}_h\in A_p(\mathbb R^n)$, $1 < p < \infty$, if and only if $-n < \alpha < n(p-1)$.
\end{itemize}
\end{proposition}
Let us give the following standard properties of $A_p$ weights which are used in the sequel.
\begin{proposition}\label{pro2.3DFan}
Let $\omega\in A_p(\mathbb R^n) \cap RH_r(\mathbb R^n)$, $p\geq 1$ and $r > 1$. Then there exist constants $C_1, C_2 > 0$ such that
$
C_1\big(\frac{|E|}{|B|}\big)^p\leq \frac{\omega(E)}{\omega(B)}\leq C_2\big(\frac{|E|}{|B|}\big)^{\frac{(r-1)}{r}},$ for any measurable subset $E$  of a ball $B$.
\end{proposition}
\begin{proposition}\label{pro2.4DFan}
If $\omega\in A_p(\mathbb R^n)$, $1 \leq p < \infty$, then for any $f\in L^1_{\rm loc}(\mathbb R^n)$ and any ball $B \subset \mathbb R^n$, we have
$
\frac{1}{|B|}\int_{B}|f(x)|dx\leq C\big(\frac{1}{\omega(B)}\int_{B}|f(x)|^p\omega(x)dx\big)^{\frac{1}{p}}.
$
\end{proposition}
%%%%%%%%%%%%%%%%%%%%%%%%%%%%%%%%
%%%%%%%%%%%%%%%%%%%%%%%%%%%%%%%
As usual, the weighted function $\omega$ is a non-negative measurable function on $\mathbb R^n$. Let $L^q(\omega)$ $(0<q<\infty)$ be the space of all measurable functions $f$ on $\mathbb{R}^n$ such that $\|f\|_{L^q(\omega)}=\big(\int_{\mathbb{R}^n}|f(x)|^q\omega(x)dx \big)^{\frac{1}{q}}<\infty.$
\\
The space $L^q_\text {loc}(\omega)$ is defined as the set of all measurable functions $f$ on $\mathbb R^n$ satisfying $\int_{K}|f(x)|^q\omega(x)dx<\infty$ for any compact subset $K$ of $\mathbb R^n$. The space $L^q_\text {loc}(\omega, \mathbb R^n\setminus\{0\})$ is also defined in a similar way to the space  $L^q_\text {loc}(\omega)$.
\vskip 5pt
In what follows, we denote $\chi_k=\chi_{C_k}$, $C_k=B_k\setminus B_{k-1}$ for all $k\in\mathbb Z$, where $B_R = \big\{x\in \mathbb R^n: |x|\leq 2^R\big\}$ and $B_R^* = \big\{x\in \mathbb R^n: |x|\leq R\big\}$ for all $R\in\mathbb R$. Denote by $\omega(K)$ the integral $\int_{K}\omega(x)dx$ for all subsets $K$ of $\mathbb R^n$.
\vskip 5pt
Now, we are in a position to give some notations and definitions of  the homogeneous  two weighted Herz spaces and the homogeneous two weighted Herz-type Hardy spaces.
\begin{definition} Let $0<\alpha<\infty, 1\leq q<\infty$, $0<p <\infty$, and let $\omega_1$ and $\omega_2$ be weighted functions. Then the homogeneous two weighted Herz space  ${\mathop{K}\limits^{.}}_{q}^{\alpha,p}(\omega_1,\omega_2)$ is defined as the set of all measurable functions $f\in L^q_{\rm loc}(\omega_2,\mathbb R^n\setminus \{0\})$  such that $\|f\|_{{\mathop{K}\limits^{.}}_{q}^{\alpha,p}(\omega_1,\omega_2)}<\infty$,
where
$$
\|f\|_{{\mathop{K}\limits^{.}}_{q}^{\alpha,p}(\omega_1,\omega_2)}=\Big( \sum\limits_{k=-\infty}^{\infty} \omega_1(B_k)^{\alpha p/n}\|f\chi_k\|_{L^q(\omega_2)}^{p}\Big)^{1/p}.
$$
\end{definition}
%%%%%%%%%%%%%%%%%%%%%%%%%%%%%%%%%%%%%%%%
%%%%%%%%%%%%%%%%%%%
%%%%%%%%%%%%%%%%%%%%%
Denote $\mathbb Z_{+}=\mathbb N\cup\lbrace0\rbrace$. Let $S(\mathbb R^n)$ be the space of Schwartz functions, and denote by $S'(\mathbb R^n)$ the dual space of $S(\mathbb R^n)$. Given $N\in\mathbb N$, we denote
$$
S_N(\mathbb R^n)=\lbrace \phi\in S(\mathbb R^n): \|\phi\|_{m,\beta}\leq 1, m\leq n+N,|\beta|\leq N\rbrace,
$$
where $\|\phi\|_{m,\beta}=\mathop{\rm sup}\limits_{x\in\mathbb R^n} (1+|x|^m)|D^{\beta}\phi(x)|,\beta=(\beta_1, ..., \beta_n), D^{\beta}\phi=(\frac{\partial}{\partial x_1})^{\beta_1}...(\frac{\partial}{\partial x_n})^{\beta_n}\phi.$ Next, the grand maximal function of $f\in S'(\mathbb R^n)$ \cite{FS1972} is defined by
$$
G_N(f)(x)=\mathop{\rm sup}\limits_{\phi\in S_N}M_\phi(f)(x), \;\;x\in\mathbb R^n,
$$ 
where $M_\phi(f)(x)=\mathop{\rm sup}\limits_{|y-x|<t}|\phi_t*f(y)|$ and $\phi_t(x)=t^{-n}\phi(t^{-1}x)$ for all $t>0$. Let us recall the definition of the Hardy spaces associated to the two weighted Herz spaces due to Lu and Yang \cite{LY1995(1)} as follows.
\begin{definition}
Let $0<\alpha<\infty, 1\leq q <\infty, 0<p<\infty$, $N={\rm max}\lbrace [\alpha-n(1-1/q)]+1,1\rbrace$ and $\omega_1,\omega_2\in A_1$. The homogeneous two weighted Herz-type Hardy space $H{\mathop{K}\limits^.}^{\alpha, p}_q(\omega_1,\omega_2)$ is defined as the set of all $f\in S'(\mathbb R^n)$ such that
$
\|f\|_{H{\mathop{K}\limits^.}^{\alpha, p}_q(\omega_1,\omega_2)}=\|G_N(f)\|_{{\mathop{K}\limits^.}^{\alpha, p}_q(\omega_1,\omega_2)}<\infty.
$
\end{definition}
Now we state the definition of central atom and dyadic central unit. Note that we denote the integer part of real number $x$ by $[x]$.
\begin{definition}
Let $1<q<\infty$, $\alpha\in[n(1-1/q),\infty)$, $s\geq [\alpha-n(1-1/q)]$ and $\omega_1,\omega_2\in A_1$.  A function $a$ on $\mathbb R^n$ is called a central $(\alpha, q, s; \omega_1, \omega_2)_0$-atom if it satisfies that
\begin{itemize}
\item[(i)] ${\rm supp}a\subset B(0,r)$ for some $r>0$;
\\
\item[(ii)] $\|a\|_{L^q(\omega_2)}\leq \omega_1(B(0,r))^{\frac{-\alpha}{n}}$;
\\
\item[(iii)] $\int_{\mathbb R^n}a(x)x^{\beta}dx=0$ for all $|\beta|\leq s$;
\\
\item[(iv)] $a(x)=0$ on some neighborhood of $0$.
\end{itemize}
 A function $a$ on $\mathbb R^n$ is called a dyadic central $(\alpha, q; \omega_1, \omega_2)$-unit if it satisfies $(\rm i)$ and $(\rm ii)$ associated to $r=2^{k}$ for some $k\in\mathbb Z$.
\end{definition}
%%%%%%%%%%%%%%%%
%%%%%%%%%%%%%%%%%%
\begin{theorem}{\rm (Theorem 1.1 in \cite{LY1995})}\label{blockHerz}
Let $0 <\alpha < \infty, 0 < p <\infty, 1\leq q <\infty$. Let $\omega\in A_1(\mathbb R^n)$ and $\omega_2$ be a weighted function on $\mathbb R^n$. We then have  $f\in {\mathop{K}\limits^{.}}^{\alpha,p}_{q}(\omega_1,\omega_2)$  if and only if
\[
f=\sum\limits_{k=-\infty}^{\infty}\lambda_k b_k\;\;\;\,\,\textit{pointwise},
\]
where $\sum\limits_{k=-\infty}^{\infty}{|\lambda_k|^p}<\infty$, and each $b_k$ is a dyadic central $(\alpha, q, \omega_1, \omega_2)$- unit with the support in $B_k$. Moreover,
$
\|f\|_{{\mathop{K}\limits^{.}}^{\alpha,p}_{q}(\omega_1,\omega_2)}\simeq \textit{\rm inf}\Big\{\Big(\sum\limits_{k=-\infty}^{\infty}|\lambda_k|^p\Big)\Big\}^{\frac{1}{p}},
$where the infimum is taken over all decompositions of f as above.
\end{theorem}
%%%%%%%%%%%%%%%
%%%%%%%%%%%%%%
Next, we present the useful result due to Zhou in \cite{Z2009} which states that the norms in $H{\mathop{K}\limits^{.}}^{\alpha,p}_{q}(\omega_1,\omega_2)$ can be achieved by finite central atomic decomposition in some dense subspaces of them.

Let $0<p<\infty$, $1<q<\infty$, $\alpha\in [n(1-1/q),\infty)$, $s\geq [\alpha-n(1-1/q)]$ and $\omega_1,\omega_2\in A_1$. Denote by ${\mathop{F}\limits^{.}}^{\alpha, q, s}_p(\omega_1,\omega_2)$ the collection of all finite linear combinations of central $(\alpha, q, s; \omega_1, \omega_2)_0$-atoms. Then for $f\in {\mathop{F}\limits^{.}}^{\alpha, q, s}_p(\omega_1,\omega_2)$, we define
\begin{align}\label{Fcentral}
\|f\|_{{\mathop{F}\limits^{.}}^{\alpha, q, s}_p(\omega_1,\omega_2)} = &{\rm inf}\Big\lbrace \big(\sum\limits_{j=1}^m |\lambda_j|^p\big)^{\frac{1}{p}}:m\in\mathbb N, f=\sum\limits_{j=1}^m \lambda_ja_j,\nonumber
\\
\;\;\;\;\;\;\;\;\;\;\;\;\;\;\;\;\;\;\;\;\,\;\;\;\;\;\;\;\;\;\;\;\;\;\;\;& \lbrace a_j\rbrace_{j=1}^m\,\textit{\rm are central}(\alpha, q, s; \omega_1, \omega_2)_0-\,\textit{\rm atoms}\Big\rbrace.
\end{align}
Let $C{\mathop{F}\limits^{.}}^{\alpha, q, s}_p(\omega_1,\omega_2)$ be the collection of all finite linear combinations of $\mathcal C^{\infty}(\mathbb R^n)$ central $(\alpha, q, s; \omega_1,\omega_2)_0$-atoms. Similarly,  for $f\in C{\mathop{F}\limits^{.}}^{\alpha, q, s}_p(\omega_1,\omega_2)$, we also define $\|f\|_{C{\mathop{F}\limits^{.}}^{\alpha, q, s}_p(\omega_1,\omega_2)}$ as in (\ref{Fcentral}) just replacing central $(\alpha, q, s; \omega_1,\omega_2)_0$-atoms by $\mathcal C^{\infty}(\mathbb R^n)$ central $(\alpha, q, s; \omega_1, \omega_2)_0$-atoms.
%%%%%%%%%%%%%%%%%%%%%%%
%%%%%%%%%%%%%%%%%%%%%%%
\begin{theorem}{\rm (Theorem 1 in \cite{Z2009})}\label{equiv}
Let $0<p<\infty$, $1<q<\infty$, $\alpha\in[n(1-1/q),\infty)$ and non-negative integer $s\geq [\alpha-n(1-1/q)]$. Then we have $\|\cdot\|_{H{\mathop{K}\limits^{.}}^{\alpha, p}_q(\omega_1,\omega_2)}$ and $\|\cdot\|_{{\mathop{F}\limits^{.}}^{\alpha, q, s}_p(\omega_1,\omega_2)}$\big(resp. $\|\cdot\|_{C{\mathop{F}\limits^{.}}^{\alpha, q, s}_p(\omega_1,\omega_2)}$\big) are equivalent on ${\mathop{F}\limits^{.}}^{\alpha, q, s}_p(\omega_1,\omega_2)$ \big(resp. $C{\mathop{F}\limits^{.}}^{\alpha, q, s}_p(\omega_1,\omega_2)$\big).
\end{theorem}
To end this section, let us recall that a quasi-Banach space $\mathcal B$ is a vector space endowed with a quasi-norm $\|\cdot\|_{\mathcal B}$ which is nonnegative, non-degenerate, homogeneous, and obeys the quasi-triangle inequality. Let $p\in (0,1]$. A quasi-Banach space $\mathcal B_p$ with a quasi-norm $|\cdot\|_{\mathcal B_p}$ is said to be a $p$-quasi-Banach space if $\|f+g\|_{\mathcal B_p}^{p}\leq \|f\|_{\mathcal B_p}^p+\|g\|_{\mathcal B_p}^p$, for any $f, g\in \mathcal B_p$.
\vskip 5pt
Recall that for any given $r$-quasi-Banach space $\mathcal B_r$ with $r \in (0, 1]$ and linear space $X$, an operator $T$ from $X$ to $B_r$ is called to be $B_r$-sublinear if for any $f, g \in X$
and $\lambda, \nu \in\mathbb C$, we have
$$
\|T(\lambda f+\nu g)\|_{\mathcal B_r}\leq (|\lambda|^r\|T(f)\|_{\mathcal B_r}^r + |\nu|^r\|T(g)\|_{\mathcal B_r}^r)^{\frac{1}{r}},
$$
 and
$$\|T(f)-T(g)\|_{\mathcal B_r}\leq \|T(f-g)\|_{\mathcal B_r}.$$
%%%%%%%%%%%%%%
%%%%%%%%%%%%%%%%
Let us give the following useful which is used in the sequel.
\begin{theorem}{\rm (Theorem 2 in \cite{Z2009})}\label{apply}
Let $0<p\leq 1$, $p\leq r\leq 1$, $1<q<\infty$, $\alpha\in [n(1-1/q),\infty)$ and nonnegative $s\geq [\alpha-n(1-1/q)]$. If $T$ is a $\mathcal B_r$-sublinear operator defined on ${\mathop{F}\limits^.}_p ^{\alpha, q, s}(\omega_1, \omega_2)$ such that
$$
S ={\rm sup}\big\lbrace \|Ta\|_{\mathcal B_r}: a\,\textit{ is any central}\, (\alpha, q, s; \omega_1,\omega_2)_0\,\textit{ -atom}\big\rbrace <\infty
$$
or defined on $C{\mathop{F}\limits^.}_p ^{\alpha, q, s}(\omega_1, \omega_2)$   such that
$$
S ={\rm sup}\big\lbrace \|Ta\|_{\mathcal B_r}: a\,\textit{ is any }\, \mathcal C^{\infty}(\mathbb R^n)\,\textit{ central}\, (\alpha, q, s; \omega_1,\omega_2)_0\,\textit{-atom}\big\rbrace <\infty,
$$
then $T$ uniquely extends to be a bounded $\mathcal B_r$-sublinear operator from $H{\mathop{K}\limits^{.}}^{\alpha,p}_{q}(\omega_1,\omega_2)$ to $\mathcal B_r$.
\end{theorem}
%%%%%%%%%%%%%%%%%%%%%%%%%%%%%%%%%%%%%%%
%%%%%%%%%%%%%%%%%%%%%%%%%%%%%%%%%%%%%%%%%%
\section{The main results about the boundedness of ${\mathcal{H}}_{\Phi,\Omega}$}\label{3}
 
%Let us recall that the rough Hausdorff operator is defined by
%$$ {\mathcal {H}}_{\Phi,\Omega}(f)(x)=\int_{\mathbb R^n}\dfrac{\Phi(x|y|^{-1})}{|y|^n}\Omega(y')f(y)dy,\;\;x\in\mathbb R^n, $$
%where $y'=\frac{y}{|y|}$, $\Phi$ is a locally integrable and radial  function on $\mathbb R^n$ and $\Omega:S_{n-1}\longrightarrow\mathbb C $ is a measurable function such that $\Omega(y)\neq 0$ for almost everywhere $y$ in $S_{n-1}$. We write ${\mathcal {H}}_{\Phi}:={\mathcal {H}}_{\Phi,\Omega}$ for $\Omega=1$.
%\vskip 5pt 
Our first main result is the following.
\begin{theorem}\label{Hardy-Hardy}
Let $0<p\leq 1$, $1<q<\infty$, $n(1-1/q)\leq\alpha<1+n(1-1/q)$, and $\omega_1=|x|^{\beta_1}$, $\omega_2=|x|^{\beta_2}$ with $\beta_1,\beta_2\in (-n,0]$. If  $\Phi$ is a radial function and there exist $m, M\in \mathbb Z$ such that ${\rm supp}(\Phi)\subset \lbrace x\in\mathbb R^n: 2^m< |x|\leq 2^M \rbrace$, then $\mathcal H_{\Phi}$ is a bounded operator from $H{\mathop{K}\limits^{.}}^{\alpha,p}_{q}(\omega_1,\omega_2)$ into itself. 
\end{theorem}
%%%%%%%%%%%%%%%%%%%%%%%%

\begin{proof}
 Let $a$ be any central $(\alpha, q,0;\omega_1,\omega_2)_0$-atom. Thus, there exits $j_a\in\mathbb Z$ such that ${\rm supp} (a)\subset B_{j_a}$ and $\|a\|_{L^q(\omega_2)}\lesssim \omega_1(B_{j_a})^{\frac{-\alpha}{n}}$. We will prove that 
\begin{align}\label{uniform}
\|\mathcal H_{\Phi}(a)\|_{H{\mathop{K}\limits^{.}}^{\alpha,p}_{q}(\omega_1,\omega_2)}\lesssim \int_{(2^m,2^M]}|\Phi(t)|dt.
\end{align}
Indeed, we can rewrite
$
\,\mathcal H_{\Phi}(a)(x)= \sum\limits_{k=m+1}^M\,\int_{(2^{k-1}, 2^k]}\int_{S_{n-1}} \dfrac{\Phi(t)}{t}a(t^{-1}|x|y')dy'dt:= \sum\limits_{k=m+1}^M b_k(x).
$
\\
From the definition of $b_k$ and ${\rm supp}(a)\subset B_{j_a}$, it is clear to see that 
\begin{align}\label{suppbk}
{\rm supp}(b_k)\subset B_{j_a+k}.
\end{align}
By the Minkowski inequality and the H\"{o}lder inequality, we have
\begin{align}
\|b_k\|_{L^q(\omega_2)}&=\Big(\int_{B_{j_a+k}}\Big|\int_{(2^{k-1}, 2^k]}\int_{S_{n-1}} \dfrac{\Phi(t)}{t}a(t^{-1}|x|y')dy'dt\Big|^q\omega_2(x)dx\Big)^{\frac{1}{q}}\nonumber
\\
&\leq \int_{(2^{k-1}, 2^k]}\frac{|\Phi(t)|}{t}\int_{S_{n-1}}\|a(|\cdot|t^{-1}y')\|_{L^q(\omega_2, B_{j_a+k})}dy'dt\nonumber
\\
&\lesssim \int_{(2^{k-1}, 2^k]}\frac{|\Phi(t)|}{t}\Big(\int_{S_{n-1}}\int_{B_{j_a+k}}|a(|x|t^{-1}y')|^q\omega_2(x)dxdy'\Big)^{\frac{1}{q}}dt.\nonumber
\end{align}
By polor coordinates, we calculate
\begin{align}\label{esaLq}
&\int_{S_{n-1}}\int_{B_{j_a+k}}|a(|x|t^{-1}y')|^q\omega_2(x)dxdy'
= \int_{S_{n-1}}\int_{0}^{2^{j_a+k}}\int_{S_{n-1}}|a(|rx'|t^{-1}y')|^qr^{n-1}|rx'|^{\beta_2}dx'drdy'\nonumber
\\
&\lesssim \int_{S_{n-1}}\int_{0}^{2^{j_a+k}}|a(rt^{-1}y')|^qr^{n-1+\beta_2}drdy'=t^{n+\beta_2}\int_{S_{n-1}}\int_{0}^{t^{-1}.2^{j_a+k}}|a(s.y')|^q s^{n-1+\beta_2}dsdy'\nonumber
\\
&= t^{n+\beta_2}\|a\|^q_{L^q(\omega_2, B^*_{ t^{-1}2^{j_a+k}})}\leq t^{n+\beta_2}\|a\|^q_{L^q(\omega_2)}.
\end{align}
Thus, by $\|a\|_{L^q(\omega_2)}\lesssim \omega_1(B_{j_a})^{\frac{-\alpha}{n}}$, it immediately follows that
\begin{align}\label{bkLq}
\|b_k\|_{L^q(\omega_2)}&\lesssim \Big(\int_{(2^{k-1},2^k]}\frac{|\Phi(t)|}{t^{1-\frac{(n+\beta_2)}{q}}}dt\Big)\|a\|_{L^q(\omega_2)}\lesssim \Big(\int_{(2^{k-1},2^k]}\frac{|\Phi(t)|}{t^{1-\frac{(n+\beta_2)}{q}}}dt\Big)\omega_1(B_{j_a+k})^{\frac{-\alpha}{n}}\Big(\frac{\omega_1(B_{j_a+k})}{\omega_1(B_{j_a})}\Big)^{\frac{\alpha}{n}}.
\end{align}
In addition, because $\omega_1$ is a power weighted function, one has
\begin{align}\label{omega1power}
\Big(\frac{\omega_1(B_{j_a+k})}{\omega_1(B_{j_a})}\Big)^{\frac{\alpha}{n}}\simeq \Big(\frac{2^{(j_a+k)(\beta_1+n)}}{2^{j_a(\beta_1+n)}}\Big)^{\frac{\alpha}{n}}=2^{k(\alpha+\frac{\beta_1\alpha}{n})}.
\end{align}
As reason above, by letting $t\in(2^{k-1},2^k]$, we have
\begin{align}\label{normbk}
\|b_k\|_{L^q(\omega_2)}&\lesssim  \Big(\int_{(2^{k-1},2^k]}\frac{|\Phi(t)|}{t^{1-\frac{(n+\beta_2)}{q}-\alpha-\frac{\beta_1\alpha}{n}}}dt\Big)\omega_1(B_{j_a+k})^{\frac{-\alpha}{n}}
:=\lambda_k.\omega_1(B_{j_a+k})^{\frac{-\alpha}{n}},
\end{align}
 where 
$$
\lambda_k=\int_{(2^{k-1},2^k]}\frac{|\Phi(t)|}{t^{1-\frac{(n+\beta_2)}{q}-\alpha-\frac{\beta_1\alpha}{n}}}dt.
$$
Next, by $\omega_2=|x|^{\beta_2}\in A_1$ and ${\rm supp}(b_k)\subset B_{j_a+k}$, we have $b_k\in L^1(\mathbb R^n)$. Combining this together with $\int_{\mathbb R^n}a(x)dx=0$, by polar coordinates and the Fubini theorem, it is easy to check that
\begin{align}\label{intbk=0}
\int_{\mathbb R^n}b_k(x)dx=|S_{n-1}|\int_{(2^{k-1},2^{k}]}\Phi(t)t^{n-1} dt.\int_{\mathbb R^n} a(y)dy=0.
\end{align}
Since there exists $r_a\in\mathbb Z$ such that $a=0$ on $B_{r_a}$, we have
\begin{align}\label{neighbk}
b_k=0\,\;\textit{ \rm on }\,\; B_{r_a+k-1}.
\end{align}
For convenience, we denote
$
b_{j_a,k}= \left\{ \begin{array}{l}
\dfrac{b_k}{\lambda_k},\textit{ \rm if }\, \lambda_k\neq 0,
\\
0,\,\,\,\,\textit\,{ \rm otherwise }.\nonumber
\end{array} \right.
$ 
\\
This shows that
$$
\mathcal H_{\Phi}(a)(x)= \sum\limits_{k=m+1}^M \lambda_k. b_{j_a,k}(x). 
$$
By (\ref{suppbk}), (\ref{normbk}), (\ref{intbk=0}) and (\ref{neighbk}), we have $b_{j_a,k}$ is central $(\alpha, q,0; \omega_1,\omega_2)_0$-atom. This implies that 
$
\mathcal H_{\Phi}(a)\in {\mathop{F}\limits^.}^{\alpha, q, 0}_p(\omega_1,\omega_2).
$ Consequently, by Theorem \ref{equiv}, one has
\begin{align}
\|\mathcal H_{\Phi}(a)\|_{H{\mathop{K}\limits^.}^{\alpha, p}_q(\omega_1,\omega_2)}&\lesssim \|\mathcal H_{\Phi}(a)\|_{{\mathop{F}\limits^.}^{\alpha, q, 0}_p(\omega_1,\omega_2)}\leq \Big(\sum\limits_{k=m+1}^M |\lambda_k|^p\Big)^{\frac{1}{p}}\lesssim \sum\limits_{k=m+1}^M |\lambda_k|,\nonumber
\end{align}
which implies that the inequality (\ref{uniform}) is true. From this, by Theorem \ref{apply}, we conclude that
$$
\|\mathcal H_{\Phi}(f)\|_{H{\mathop{K}\limits^{.}}^{\alpha,p}_{q}(\omega_1,\omega_2)}\lesssim \Big(\int_{(2^m, 2^M]}|\Phi(t)|dt\Big).\|f\|_{H{\mathop{K}\limits^{.}}^{\alpha,p}_{q}(\omega_1,\omega_2)},
$$
 for all $f\in H{\mathop{K}\limits^{.}}^{\alpha,p}_{q}(\omega_1,\omega_2)$.
Therefore, the proof of this theorem is finished.
\end{proof}
%%%%%%%%%%%%%%%%%%
%%%%%%%%%%%%%%%%%%%%
\begin{theorem}\label{Hardy-Herz}
Let $1<q<\infty, \alpha\in [n(1-1/q),\infty)$, $\omega_1=|x|^{\beta_1}$, $\omega_2=|x|^{\beta_2}$ with $\beta_1,\beta_2\in (-n,0]$. Let $\Omega\in L^{q'}(S_{n-1})$ and $\Phi$ be a radial function.
\\
{\rm (i)} If $p=1$ and 
$$
\mathcal C_1=\int_{0}^{\infty}\frac{|\Phi(t)|}{t^{1-\frac{(n+\beta_2)}{q}-\alpha-\frac{\beta_1\alpha}{n}}}dt<\infty,
$$
we then have
$
\|\mathcal H_{\Phi,\Omega}(f)\|_{{\mathop{K}\limits^{.}}^{\alpha,1}_{q}(\omega_1,\omega_2)}\lesssim C_1\|\Omega\|_{L^{q'}(S_{n-1})}\mathcal \|f\|_{H{\mathop{K}\limits^{.}}^{\alpha, 1}_{q}(\omega_1,\omega_2)},\,\textit{for all}\, f\in H{\mathop{K}\limits^{.}}^{\alpha, 1}_{q}(\omega_1,\omega_2).
$
\\
{\rm (ii)} If $0<p<1$, $\sigma>\frac{1-p}{p}$ and
$$
\mathcal C_2= \int_{0}^{\infty}\frac{|\Phi(t)|}{t^{1-\frac{(n+\beta_2)}{q}-\alpha-\frac{\beta_1\alpha}{n}}}\Big(\chi_{(0,1]}(t).|{\rm log}_2(t)|^{\sigma} + \chi_{(1,\infty)}(t).\big({\rm log}_2(t) +1 \big)^{\sigma}\Big)dt<\infty,
$$
we then have
$
\|\mathcal H_{\Phi,\Omega}(f)\|_{{\mathop{K}\limits^{.}}^{\alpha,p}_{q}(\omega_1,\omega_2)}\lesssim C_2\|\Omega\|_{L^{q'}(S_{n-1})}\mathcal \|f\|_{H{\mathop{K}\limits^{.}}^{\alpha, p}_{q}(\omega_1,\omega_2)},\,\textit{for all}\, f\in H{\mathop{K}\limits^{.}}^{\alpha, p}_{q}(\omega_1,\omega_2).
$
\end{theorem}
%%%%%%%%%%%%%%%%%%%%%
%%%%%%%%%%%%%%%%%%%%%
\begin{proof}
Let us fix non-negative integer $s\geq [\alpha-n(1-1/q)]$, and let $a$ be any central $(\alpha, q, s;\omega_1, \omega_2)_0$-atom with ${\rm supp}(a)\subset B_{j_a}$ and $\|a\|_{L^q(\omega_2)}\lesssim \omega_1(B_{j_a})^{\frac{-\alpha}{n}}$.
We need to prove that
\begin{align}\label{uniform1}
\|\mathcal H_{\Phi,\Omega}(a)\|_{{\mathop{K}\limits^{.}}^{\alpha,p}_{q}(\omega_1,\omega_2)}\lesssim \left\{ \begin{array}{l}
\mathcal C_1.\|\Omega\|_{L^{q'}(S_{n-1})},\,\,p=1,
\\
\mathcal C_2.\|\Omega\|_{L^{q'}(S_{n-1})},\,\,p\in (0,1)\,\textit{\rm and}\, \sigma>\frac{(1-p)}{p}.
\end{array} \right.
\end{align}
Infact, we have
\begin{align}\label{xichmabk}
|\mathcal H_{\Phi,\Omega}(a)(x)|&\leq\sum\limits_{k\in\mathbb Z}\,\int_{(2^{k-1}, 2^k]}\int_{S_{n-1}} \dfrac{|\Phi(t)|}{t}|\Omega(y')|.|a(t^{-1}|x|y')|dy'dt:=  \sum\limits_{k\in\mathbb Z} \widetilde{b}_k(x).
\end{align}
By estimating as (\ref{suppbk}) and (\ref{normbk}) above, it is obvious to see that 
\begin{align}\label{suppnorm}
{\rm supp}(\widetilde{b}_k)\subset B_{j_a+k}\,\textit{ \rm and }\,\|\widetilde{b}_k\|_{L^q(\omega_2)}\lesssim \lambda_k.\|\Omega\|_{L^{q'}(S_{n-1})}.\omega_1(B_{j_a+k})^{\frac{-\alpha}{n}}.
\end{align}
Denote $\widetilde{b}_{j_a,k}$ as follows
$
\widetilde{b}_{j_a,k} = \left\{ \begin{array}{l}
\dfrac{\widetilde{b}_k}{\|\Omega\|_{L^{q'}(S_{n-1})}.\lambda_k},\textit{ \rm if }\, \lambda_k\neq 0,
\\
0,\,\textit\,\,\,\,\,\,\,\,\,\,\,\,\,\,\,\,\,\,\,\,\,\,\,\,\,\,\,\,\,\,{ \rm otherwise }.\nonumber
\end{array} \right.
$
\\
This gives 
$$
\sum\limits_{k\in\mathbb Z} \widetilde{b}_k= \sum\limits_{k\in\mathbb Z}\|\Omega\|_{L^{q'}(S_{n-1})}.\lambda_k.\widetilde{b}_{j_a,k}.
$$ 
By (\ref{suppnorm}), we see that $\widetilde{b}_{j_a,k}$ is a dyadic central $(\alpha, q,\omega_1,\omega_2)$-unit. Therefore, by Theorem \ref{blockHerz}, we infer
\begin{align}\label{lambdakHardyHerz}
\|\mathcal H_{\Phi, \Omega}(a)\|_{{\mathop{K}\limits^{.}}^{\alpha,p}_{q}(\omega_1,\omega_2)}\leq \|\sum\limits_{k\in\mathbb Z} \widetilde{b}_k\|_{{\mathop{K}\limits^{.}}^{\alpha,p}_{q}(\omega_1,\omega_2)}\lesssim \|\Omega\|_{L^{q'}(S_{n-1})}\Big(\sum\limits_{k\in\mathbb Z} |\lambda_k|^p\Big)^{\frac{1}{p}}.
\end{align}
For $p=1$, we estimate
\begin{align}\label{C1HardyHerz}
\sum\limits_{k\in\mathbb Z} |\lambda_k| =\sum\limits_{k\in\mathbb Z}\,\int_{(2^{k-1},2^k]}\frac{|\Phi(t)|}{t^{1-\frac{(n+\beta_2)}{q}-\alpha-\frac{\beta_1\alpha}{n}}}dt=\mathcal C_1.
\end{align}
For $p\in (0,1)$ and $\sigma>\frac{1-p}{p}$, by the H\"{o}lder inequality, we get
\begin{align}\label{lambdakp}
&\Big(\sum\limits_{k\in\mathbb Z} |\lambda_k|^p\Big)^{\frac{1}{p}}\lesssim \sum\limits_{k\in\mathbb Z} |k|^{\sigma}|\lambda_k|= \sum\limits_{k=1}^{\infty} \,\int_{(2^{k-1},2^k]}\frac{|\Phi(t)|}{t^{1-\frac{(n+\beta_2)}{q}-\alpha-\frac{\beta_1\alpha}{n}}}|k|^{\sigma}dt+\sum\limits_{k=-\infty}^0 \,\int_{(2^{k-1},2^k]}\frac{|\Phi(t)|}{t^{1-\frac{(n+\beta_2)}{q}-\alpha-\frac{\beta_1\alpha}{n}}}|k|^{\sigma}dt\nonumber
\\
&\lesssim \sum\limits_{k=1}^{\infty}\,\int_{(2^{k-1},2^k]}\frac{|\Phi(t)|}{t^{1-\frac{(n+\beta_2)}{q}-\alpha-\frac{\beta_1\alpha}{n}}}({\rm log}_2(t)+1)^{\sigma}dt +\sum\limits_{k=-\infty}^0\,\int_{(2^{k-1},2^k]}\frac{|\Phi(t)|}{t^{1-\frac{(n+\beta_2)}{q}-\alpha-\frac{\beta_1\alpha}{n}}}|{\rm log}_2(t)|^{\sigma}dt= \mathcal C_2.
\end{align}
From this, by (\ref{lambdakHardyHerz}) and (\ref{C1HardyHerz}), the inequality (\ref{uniform1}) is proved. Combining Theorem \ref{apply} and the inequality (\ref{uniform1}), we finish the proof of this theorem.
\end{proof}
%%%%%%%%%%%%%%%%%%%%%%%%
%%%%%%%%%%%%%%%%%%%%%%%%
%%%%%%%%%%%%%%%%%%
%%%%%%%%%%%%%%%%%%%%
\begin{theorem}\label{Hardy-Herz1}
Suppose $1<q<\infty, \alpha\in [n(1-1/q),\infty)$, $\omega_2=|x|^{\beta_2}$ with $\beta_2\in (-n,0]$. Let $\Omega\in L^{q'}(S_{n-1})$ and $\Phi$ be a radial function. At the same time, let $\omega_1\in A_1$ with the finite critical index $r_{\omega_1}$ for the reverse H\"{o}lder condition and $\delta\in (1,r_{\omega_1})$.
\\
{\rm (i)} If $p=1$ and 
$$
\mathcal C_3=\int_{0}^{\infty}\frac{|\Phi(t)|}{t^{1-\frac{(n+\beta_2)}{q}}}\Big(\chi_{(0,1]}(t).t^{\frac{\alpha(\delta-1)}{\delta}}+ \chi_{(1,\infty)}(t).t^{\alpha}\Big)dt<\infty,
$$
then we have
$
\|\mathcal H_{\Phi,\Omega}(f)\|_{{\mathop{K}\limits^{.}}^{\alpha,1}_{q}(\omega_1,\omega_2)}\lesssim C_3.\|\Omega\|_{L^{q'}(S_{n-1})}\mathcal .\|f\|_{H{\mathop{K}\limits^{.}}^{\alpha, 1}_{q}(\omega_1,\omega_2)},\,\,\textit{for all}\, f\in H{\mathop{K}\limits^{.}}^{\alpha, 1}_{q}(\omega_1,\omega_2).
$
\\
{\rm (ii)} If $0<p<1$, $\sigma>\frac{1-p}{p}$ and
$$
\mathcal C_4= \int_{0}^{\infty}\frac{|\Phi(t)|}{t^{1-\frac{(n+\beta_2)}{q}}}\Big(\chi_{(0,1]}(t).t^{\frac{\alpha(\delta-1)}{\delta}}|{\rm log}_2(t)|^{\sigma} + \chi_{(1,\infty)}(t).t^{\alpha}\big({\rm log}_2(t) +1 \big)^{\sigma}\Big)dt<\infty,
$$
then we have
$
\|\mathcal H_{\Phi,\Omega}(f)\|_{{\mathop{K}\limits^{.}}^{\alpha,p}_{q}(\omega_1,\omega_2)}\lesssim C_4.\|\Omega\|_{L^{q'}(S_{n-1})}\mathcal .\|f\|_{H{\mathop{K}\limits^{.}}^{\alpha, p}_{q}(\omega_1,\omega_2)},\,\,\textit{for all}\, f\in H{\mathop{K}\limits^{.}}^{\alpha, p}_{q}(\omega_1,\omega_2).
$
\end{theorem}
%%%%%%%%%%%%%%%%%%%
%%%%%%%%%%%%%%%%%%%
\begin{proof}
Analogously to Theorem \ref{Hardy-Herz}, to prove the theorem, it suffices to show that
\begin{align}\label{uniform2}
\|\mathcal H_{\Phi,\Omega}(a)\|_{{\mathop{K}\limits^{.}}^{\alpha,p}_{q}(\omega_1,\omega_2)}\lesssim \left\{ \begin{array}{l}
\mathcal C_3.\|\Omega\|_{L^{q'}(S_{n-1})},\,\,p=1,
\\
\mathcal C_4.\|\Omega\|_{L^{q'}(S_{n-1})},\,\,p\in (0,1)\,\textit{\rm and}\,\sigma>\frac{(1-p)}{p}.
\end{array} \right.
\end{align}
where non-negative integer $s\geq [\alpha-n(1-1/q)]$ and $a$ is any central $(\alpha, q, s;\omega_1, \omega_2)_0$-atom with ${\rm supp}(a)\subset B_{j_a}$. Let us recall that 
$
\widetilde{b}_k(x)=\int_{(2^{k-1}, 2^k]}\int_{S_{n-1}} \dfrac{|\Phi(t)|}{t}|\Omega(y')|.|a(t^{-1}|x|y')|dy'dt,
$ with ${\rm supp}(\widetilde{b}_k)\subset B_{j_a+k}$. Now, by estimating as (\ref{bkLq}) above, we also have
$$
\|\widetilde{b}_k\|_{L^q(\omega_2)}\leq\|\Omega\|_{L^{q'}(S_{n-1})}.\Big(\int_{(2^{k-1},2^k]}\frac{|\Phi(t)|}{t^{1-\frac{(n+\beta_2)}{q}}}dt\Big)\omega_1(B_{j_a+k})^{\frac{-\alpha}{n}}\Big(\frac{\omega_1(B_{j_a+k})}{\omega_1(B_{j_a})}\Big)^{\frac{\alpha}{n}}.
$$
On the other hand, by $\omega_1\in A_1$ and Proposition \ref{pro2.3DFan}, we get
\begin{align}\label{omega1A1}
\Big(\frac{\omega_1(B_{j_a+k})}{\omega_1(B_{j_a})}\Big)^{\frac{\alpha}{n}}\lesssim\left\{ \begin{array}{l}
\big(\frac{|B_{j_a+k}|}{|B_{j_a}|}\big)^{\frac{\alpha}{n}}\lesssim 2^{k\alpha},\,\,\,\,\,\,\,\,\,\,\,\,\,\,\,\textit{\rm if}\,\,\, k\geq 1,
\\
\\
\big(\frac{|B_{j_a+k}|}{|B_{j_a}|}\big)^{\frac{(\delta-1)\alpha}{\delta n}}\lesssim 2^{\frac{k\alpha(\delta-1)}{\delta}},\,\textit{\rm otherwise.}
\end{array} \right.
\end{align}
Thus, by letting $t\in (2^{k-1}, 2^k]$, we have
\begin{align}
\|\widetilde{b}_k\|_{L^q(\omega_2)}&\lesssim \omega_1(B_{j_a+k})^{\frac{-\alpha}{n}}.\|\Omega\|_{L^{q'}(S_{n-1})}.
\left\{ \begin{array}{l}
\int_{(2^{k-1},2^k]}\frac{|\Phi(t)|}{t^{1-\frac{(n+\beta_2)}{q}-\alpha}}dt,\,\,\,\,\,\,\,\,\,\,\,\,\,\,\,\,\textit{\rm if}\,\,k\geq 1,
\\
\\
\int_{(2^{k-1},2^k]}\frac{|\Phi(t)|}{t^{1-\frac{(n+\beta_2)}{q}-\frac{\alpha(\delta-1)}{\delta}}}dt,\,\,\textit{\rm otherwise},
\end{array} \right.
\nonumber
\\
&:= \omega_1(B_{j_a+k})^{\frac{-\alpha}{n}}.\|\Omega\|_{L^{q'}(S_{n-1})}.{\mu}_k.\nonumber
\end{align}
Denote $\widetilde{b}^*_{j_a,k}$ as follows
$
\widetilde{b}^*_{j_a,k} = \left\{ \begin{array}{l}
\dfrac{\widetilde{b}_k}{\|\Omega\|_{L^{q'}(S_{n-1})}.{\mu}_k},\,\,\textit{\rm if }\, \mu_k\neq 0,
\\
0,\,\,\,\,\,\,\,\,\,\,\,\,\,\,\,\,\,\,\,\,\,\,\,\,\,\,\,\,\,\,\,\, \textit\,{\rm otherwise}.\nonumber
\end{array} \right.
$
\\
Thus, we have
$$
\sum\limits_{k\in\mathbb Z} \widetilde{b}_k= \sum\limits_{k\in\mathbb Z}\|\Omega\|_{L^{q'}(S_{n-1})}.\mu_k.\widetilde{b}^*_{j_a,k},
$$ 
where $\widetilde{b}^*_{j_a,k}$ is a dyadic central $(\alpha, q,\omega_1,\omega_2)$-unit. From this, by the inequality (\ref{xichmabk}) and Theorem \ref{blockHerz}, one has
$
\|\mathcal H_{\Phi,\Omega}(a)\|_{{\mathop{K}\limits^{.}}^{\alpha,p}_{q}(\omega_1,\omega_2)}\lesssim \|\Omega\|_{L^{q'}(S_{n-1})}.\Big(\sum\limits_{k\in\mathbb Z} |\mu_k|^p\Big)^{\frac{1}{p}}.
$
\\
For $p=1$, it is evident to see that
$\sum\limits_{k\in\mathbb Z} |\mu_k|=\mathcal C_3$. Next, we consider for the case $p\in (0,1)$ and $\sigma>\frac{(1-p)}{p}$.  By the arguments as (\ref{lambdakp}) above, we also have
\begin{align}
&\Big(\sum\limits_{k\in\mathbb Z} |\mu_k|^p\Big)^{\frac{1}{p}}\lesssim
\sum\limits_{k\in\mathbb Z} |k|^{\sigma}|\mu_k|= \sum\limits_{k=1}^{\infty} \,\int_{(2^{k-1},2^k]}\frac{|\Phi(t)|}{t^{1-\frac{(n+\beta_2)}{q}-\alpha}}|k|^{\sigma}dt+\sum\limits_{k=-\infty}^0 \,\int_{(2^{k-1},2^k]}\frac{|\Phi(t)|}{t^{1-\frac{(n+\beta_2)}{q}-\frac{\alpha(\delta-1)}{\delta}}}|k|^{\sigma}dt\nonumber
\\
&\lesssim \sum\limits_{k=1}^{\infty}\,\int_{(2^{k-1},2^k]}\frac{|\Phi(t)|}{t^{1-\frac{(n+\beta_2)}{q}-\alpha}}({\rm log}_2(t)+1)^{\sigma}dt +\sum\limits_{k=-\infty}^0\,\int_{(2^{k-1},2^k]}\frac{|\Phi(t)|}{t^{1-\frac{(n+\beta_2)}{q}-\frac{\alpha(\delta-1)}{\delta}}}|{\rm log}_2(t)|^{\sigma}dt= \mathcal C_4.
\nonumber
\end{align}
Consequently, the inequality (\ref{uniform2}) is true. This concludes that the proof of this theorem is ended.
\end{proof}
%%%%%%%%%%%%%%%%%%%%%%%
%%%%%%%%%%%%%%%%%%%%%%%%%
\begin{theorem}\label{Hardy-Herz2}
Let $1\leq q^*<q <\infty$, $0<\alpha^*<\infty$, $\alpha\in [n(1-1/q),\infty)$. Let $\Omega\in L^{q'}(S_{n-1})$, $\Phi$ be a radial function, $\omega_i\in A_1$ with the finite critical index $r_{\omega_i}$ for the reverse H\"{o}lder condition and $\delta_i\in (1,r_{\omega_i})$, for all $i=1,2$. Assume that $q>q^*r_{\omega_2}'$ and the following conditions are true:
\begin{align}\label{qalpha}
\frac{1}{q}+\frac{\alpha}{n}=\frac{1}{q^*}+\frac{\alpha^*}{n}.
\end{align}
\begin{align}\label{omega2/omega1}
\omega_2(B_k)\lesssim \omega_1(B_k),\,\textit{ for all }\,\, k\in\mathbb Z.
\end{align}
 Denote  $\gamma_1=\frac{(\delta_2-1)}{\delta_2}.(\frac{n}{q}+\alpha)- \frac{\alpha^*}{\delta_1}$, $\gamma_2=\frac{n}{q}+\alpha+ \frac{\alpha^*}{\delta_2}$.
\\
{\rm (i)} If $p=1$ and 
$$
\mathcal C_5=\int_{0}^{\infty}\frac{|\Phi(t)|}{t}\Big(\chi_{(0,1]}(t).t^{\gamma_1}+ \chi_{(1,\infty)}(t).t^{\gamma_2}\Big)dt<\infty,
$$
then we have
$
\|\mathcal H_{\Phi,\Omega}(f)\|_{{\mathop{K}\limits^{.}}^{\alpha^*,1}_{q^*}(\omega_1,\omega_2)}\lesssim C_5.\|\Omega\|_{L^{q'}(S_{n-1})}\mathcal .\|f\|_{H{\mathop{K}\limits^{.}}^{\alpha, 1}_{q}(\omega_1,\omega_2)},\,\,\textit{for all}\, f\in H{\mathop{K}\limits^{.}}^{\alpha, 1}_{q}(\omega_1,\omega_2).
$
\\
{\rm (ii)} If $0<p<1$, $\sigma>\frac{1-p}{p}$ and
$$
\mathcal C_6= \int_{0}^{\infty}\frac{|\Phi(t)|}{t}\Big(\chi_{(0,1]}(t).t^{\gamma_1}|{\rm log}_2(t)|^{\sigma} + \chi_{(1,\infty)}(t).t^{\gamma_2}\big({\rm log}_2(t) +1 \big)^{\sigma}\Big)dt<\infty,
$$
then we have
$
\|\mathcal H_{\Phi,\Omega}(f)\|_{{\mathop{K}\limits^{.}}^{\alpha^*,p}_{q^*}(\omega_1,\omega_2)}\lesssim C_6.\|\Omega\|_{L^{q'}(S_{n-1})}\mathcal .\|f\|_{H{\mathop{K}\limits^{.}}^{\alpha, p}_{q}(\omega_1,\omega_2)},\,\,\textit{ for all }\, f\in H{\mathop{K}\limits^{.}}^{\alpha, p}_{q}(\omega_1,\omega_2).
$
\end{theorem}
%%%%%%%%%%%%%%%%%%%%
\begin{proof}
We fix non negative integer $s\geq [\alpha-n(1-1/q)]$. Then, let $a$ be any central $(\alpha, q, s;\omega_1, \omega_2)_0$-atom with ${\rm supp}(a)\subset B_{j_a}$. To prove the theorem, it suffices to show that
\begin{align}\label{uniform3}
\|\mathcal H_{\Phi,\Omega}(a)\|_{{\mathop{K}\limits^{.}}^{\alpha^*,p}_{q^*}(\omega_1,\omega_2)}\lesssim \left\{ \begin{array}{l}
\mathcal C_5.\|\Omega\|_{L^{q'}(S_{n-1})},\,\,p=1,
\\
\mathcal C_6.\|\Omega\|_{L^{q'}(S_{n-1})},\,\,p\in (0,1)\,\textit{\rm and}\,\,\sigma>\frac{(1-p)}{p}.
\end{array} \right.
\end{align}
As mentioned above, we have  
$
\widetilde{b}_k(x)=\int_{(2^{k-1}, 2^k]}\int_{S_{n-1}} \dfrac{|\Phi(t)|}{t}|\Omega(y')||a(t^{-1}|x|y')|dy'dt,
$
with ${\rm supp}(\widetilde{b}_k)\subset B_{j_a+k}$.
By the Minkowski inequality, we infer
\begin{align}
\|\widetilde{b}_k\|_{L^{q^*}(\omega_2)}&=\Big(\int_{B_{j_a+k}}\Big|\int_{(2^{k-1}, 2^k]}\int_{S_{n-1}} \dfrac{\Phi(t)}{t}.\Omega(y').a(t^{-1}|x|y')dy'dt\Big|^{q^*}\omega_2(x)dx\Big)^{\frac{1}{q^*}}\nonumber
\\
&\leq \int_{(2^{k-1}, 2^k]}\frac{|\Phi(t)|}{t}\int_{S_{n-1}}|\Omega(y')|.\|a(|\cdot|t^{-1}y')\|_{L^{q^*}(\omega_2, B_{j_a+k})}dy'dt.\nonumber
\end{align}
Because we have $q>q^*r_{\omega_2}'$, there exists $r\in (1,r_{\omega_2})$ such that $q=q^* r'$, where $r'$ is conjugate real number of $r$. Thus, by the H\"{o}lder inequality and the reverse H\"{o}lder condition, we deduce that
\begin{align}
\|a(|\cdot|t^{-1}y')\|_{L^{q^*}(\omega_2, B_{j_a+k})}&\leq \|a(|\cdot|t^{-1}y')\|_{L^{q}(B_{j_a+k})}\Big(\int_{B_{j_a+k}}\omega_2^r(x)dx\Big)^{\frac{1}{rq^*}}\nonumber
\\
&\lesssim \|a(|\cdot|t^{-1}y')\|_{L^{q}(B_{j_a+k})}|B_{j_a+k}|^{\frac{-1}{q}}\omega_2(B_{j_a+k})^{\frac{1}{q^*}}.\nonumber
\end{align}
As a consequence, by the H\"{o}lder inequality and the argument as (\ref{esaLq}) above, we have
\begin{align}
\|\widetilde{b}_k\|_{L^{q^*}(\omega_2)}&\lesssim |B_{j_a+k}|^{\frac{-1}{q}}\omega_2(B_{j_a+k})^{\frac{1}{q^*}}\int_{(2^{k-1}, 2^k]}\frac{|\Phi(t)|}{t}\int_{S_{n-1}}|\Omega(y')|. \|a(|\cdot|t^{-1}y')\|_{L^{q}(B_{j_a+k})}dy'dt
\nonumber
\\
&\leq |B_{j_a+k}|^{\frac{-1}{q}}\omega_2(B_{j_a+k})^{\frac{1}{q^*}}\|\Omega\|_{L^{q'}(S_{n-1})}\int_{(2^{k-1}, 2^k]}\frac{|\Phi(t)|}{t}\Big(\int_{S_{n-1}}\int_{B_{j_a+k}} |a(|x|t^{-1}y')|^qdxdy'\Big)^{\frac{1}{q}}dt
\nonumber
\\
&\lesssim |B_{j_a+k}|^{\frac{-1}{q}}\omega_2(B_{j_a+k})^{\frac{1}{q^*}}\|\Omega\|_{L^{q'}(S_{n-1})}\int_{(2^{k-1}, 2^k]}\frac{|\Phi(t)|}{t}t^{\frac{n}{q}}\|a\|_{L^q(B^*_{t^{-1}2^{j_a+k}})}dt.\nonumber
\end{align}
By applying Proposition \ref{pro2.4DFan} and the inequality $\|a\|_{L^q(\omega_2)}\lesssim \omega_1(B_{j_a})^{\frac{-\alpha}{n}}$, one has
\begin{align}
\|a\|_{L^q(B^*_{t^{-1}2^{j_a+k}})}&\lesssim |B^*_{t^{-1}2^{j_a+k}}|^{\frac{1}{q}}\omega_2(B^*_{t^{-1}2^{j_a+k}})^{\frac{-1}{q}}\|a\|_{L^q(\omega_2)}
\lesssim |B^*_{t^{-1}2^{j_a+k}}|^{\frac{1}{q}}\omega_2(B^*_{t^{-1}2^{j_a+k}})^{\frac{-1}{q}}\omega_1(B_{j_a})^{\frac{-\alpha}{n}}.\nonumber
\end{align}
Therefore, by $\Big(\frac{|B^*_{t^{-1}2^{j_a +k}}|}{|B_{j_a+k}|}\Big)^{\frac{1}{q}}\simeq t^{\frac{-n}{q}}$, we obtain that
\begin{align}
\|\widetilde{b}_k\|_{L^{q*}(\omega_2)}\lesssim \|\Omega\|_{L^{q'}(S_{n-1})}\int_{(2^{k-1}, 2^k]}\frac{|\Phi(t)|}{t}\omega_2(B_{j_a+k})^{\frac{1}{q^*}}\omega_2(B^*_{t^{-1}2^{j_a+k}})^{\frac{-1}{q}}\omega_1(B_{j_a})^{\frac{-\alpha}{n}}dt.\nonumber
\end{align}
From $2^{k-1}<t\leq 2^k$, one has $\omega_2(B^*_{t^{-1}2^{j_a+k}})^{\frac{-1}{q}}\leq \omega_2(B_{j_a})^{\frac{-1}{q}}$. This infers
\begin{align}\label{widebkLq*}
\|\widetilde{b}_k\|_{L^{q*}(\omega_2)}\lesssim \omega_1(B_{j_a+k})^{\frac{-\alpha^*}{n}}.\|\Omega\|_{L^{q'}(S_{n-1})}\int_{(2^{k-1}, 2^k]}\frac{|\Phi(t)|}{t}.{\mathcal U}_{\omega_1,\omega_2,k,j_a}dt,
\end{align}
where ${\mathcal U}_{\omega_1,\omega_2,k,j_a}=\omega_1(B_{j_a+k})^{\frac{\alpha^*}{n}}\omega_2(B_{j_a+k})^{\frac{1}{q^*}}\omega_2(B_{j_a})^{\frac{-1}{q}}\omega_1(B_{j_a})^{\frac{-\alpha}{n}}.$
\\
Now, by (\ref{qalpha}) and (\ref{omega2/omega1}), we have
\begin{align}
{\mathcal U}_{\omega_1,\omega_2,k,j_a}&=\frac{\omega_2(B_{j_a+k})^{\frac{1}{q^*}+\frac{\alpha^*}{n}}}{\omega_2(B_{j_a})^{\frac{1}{q}+\frac{\alpha}{n}}}.\Big(\frac{\omega_2(B_{j_a})}{\omega_1(B_{j_a})}\Big)^{\frac{\alpha}{n}}.\Big(\frac{\omega_1(B_{j_a+k})}{\omega_2(B_{j_a+k})}\Big)^{\frac{\alpha^*}{n}}\nonumber
\\
&\lesssim \Big(\frac{\omega_2(B_{j_a+k})}{\omega_2(B_{j_a})}\Big)^{\frac{1}{q}+\frac{\alpha}{n}}.\Big(\frac{\omega_2(B_{j_a})}{\omega_1(B_{j_a})}\Big)^{\frac{\alpha^*}{n}}.\Big(\frac{\omega_1(B_{j_a+k})}{\omega_2(B_{j_a+k})}\Big)^{\frac{\alpha^*}{n}}\nonumber
\\
&=\Big(\frac{\omega_2(B_{j_a+k})}{\omega_2(B_{j_a})}\Big)^{\frac{1}{q}+\frac{\alpha}{n}}.\Big(\frac{\omega_2(B_{j_a})}{\omega_2(B_{j_a +k})}\Big)^{\frac{\alpha^*}{n}}.\Big(\frac{\omega_1(B_{j_a+k})}{\omega_1(B_{j_a})}\Big)^{\frac{\alpha^*}{n}}.\nonumber
\end{align}
By making Proposition \ref{pro2.3DFan}, we have
\begin{align}\label{Uomega1omega2}
{\mathcal U}_{\omega_1,\omega_2,k,j_a} = \left\{ \begin{array}{l}
\Big(\frac{|B_{j_a+k}|}{|B_{j_a}|}\Big)^{\frac{1}{q}+\frac{\alpha}{n}}.\Big(\frac{|B_{j_a}|}{|B_{j_a+k}|}\Big)^{\frac{(\delta_2-1)\alpha^*}{\delta_2.n}}.\Big(\frac{|B_{j_a+k}|}{|B_{j_a}|}\Big)^{\frac{\alpha^*}{n}}\lesssim 2^{kn\big(\frac{1}{q}+\frac{\alpha}{n}+ \frac{\alpha^*}{\delta_2.n}\big)}=2^{k\gamma_2},
\\
\textit{\rm if}\, k\geq 1,
\\
\Big(\frac{|B_{j_a+k}|}{|B_{j_a}|}\Big)^{\big(\frac{\delta_2-1}{\delta_2}\big).\big(\frac{1}{q}+\frac{\alpha}{n}\Big)}.\Big(\frac{|B_{j_a}|}{|B_{j_a+k}|}\Big)^{\frac{\alpha^*}{n}}.\Big(\frac{|B_{j_a+k}|}{|B_{j_a}|}\Big)^{\frac{(\delta_1-1)\alpha^*}{\delta_1.n}}\lesssim 2^{kn\big(\frac{(\delta_2-1)}{\delta_2}.(\frac{1}{q}+\frac{\alpha}{n})- \frac{\alpha^*}{\delta_1.n}\big)}=2^{k\gamma_1},
\\
\textit{\rm otherwise}.
\end{array} \right.
\end{align}
By (\ref{widebkLq*}) with $2^{k-1}< t\leq 2^{k}$, we lead to
\begin{align}
\|\widetilde{b}_k\|_{L^{q^*}(\omega_2)}&\lesssim \omega_1(B_{j_a+k})^{\frac{-\alpha^*}{n}}.\|\Omega\|_{L^{q'}(S_{n-1})}.
\left\{ \begin{array}{l}
\int_{(2^{k-1},2^k]}\frac{|\Phi(t)|}{t^{1-\gamma_2}}dt,\,\textit{\rm if}\,\,k\geq 1,
\\
\\
\int_{(2^{k-1},2^k]}\frac{|\Phi(t)|}{t^{1-\gamma_1}}dt,\,\textit{\rm otherwise},
\end{array} \right.
\nonumber
\\
&:= \omega_1(B_{j_a+k})^{\frac{-\alpha^*}{n}}.\|\Omega\|_{L^{q'}(S_{n-1})}.\mu^*_k.\nonumber
\end{align}
Next, we define $\widetilde{b}^{**}_{j_a,k}$ as follows
$
\widetilde{b}^{**}_{j_a,k} = \left\{ \begin{array}{l}
\dfrac{\widetilde{b}_k}{\|\Omega\|_{L^{q'}(S_{n-1})}.\mu^*_k}\,\textit{ \rm if }\, \mu^*_k\neq 0,
\\
0,\,\,\,\,\,\,\,\,\,\,\,\,\,\,\,\,\,\,\,\,\,\,\,\,\,\,\,\,\,\, \textit\,{\rm otherwise}.\nonumber
\end{array} \right.
$
\\
Consequently, one has 
$
\sum\limits_{k\in\mathbb Z} \widetilde{b}_k= \sum\limits_{k\in\mathbb Z}\|\Omega\|_{L^{q'}(S_{n-1})}\mu^*_k.\widetilde{b}^{**}_{j_a,k},
$ where $\widetilde{b}^{**}_{j_a,k}$ is a dyadic central $(\alpha^*, q^*,\omega_1,\omega_2)$-unit. Hence, by the inequality (\ref{xichmabk}) and Theorem \ref{blockHerz}, it immediately follows that
\begin{align}
\|\mathcal H_{\Phi,\Omega}(a)\|_{{\mathop{K}\limits^{.}}^{\alpha^*,p}_{q^*}(\omega_1,\omega_2)}\lesssim \|\Omega\|_{L^{q'}(S_{n-1})}\Big(\sum\limits_{k\in\mathbb Z} |\mu^*_k|^p\Big)^{\frac{1}{p}}.\nonumber
\end{align}
By estimating as the final part of the proof of Theorem \ref{Hardy-Herz1} above, we also have 
\begin{align}
&\Big(\sum\limits_{k\in\mathbb Z} |\mu^*_k|^p\Big)^{\frac{1}{p}}\lesssim \left\{ \begin{array}{l}
\mathcal C_5,\,\,p=1,
\\
\mathcal C_6,\,\,p\in (0,1)\,\textit{\rm and}\,\,\sigma>\frac{(1-p)}{p}.
\end{array} \right.\nonumber
\end{align}
This implies that the inequality (\ref{uniform3}) is right. Hence, the proof of this theorem is finished.
\end{proof}
%%%%%%%%%%%%%%%%%%%%%%%%%%%%%%%%%%%%%%%%%%%%%%
%%%%%%%%%%%%%%%%%%%%%%%%%%%%%%%%%%%%%%%%%%%%%%
%%%%%%%%%%%%%%%%%%%%%%%%%%%%%%%%%%%%%%%%%%%%%%
\section{The main results about the boundedness of ${{H}}_{\Phi,A}$}\label{4}
 For a matrix $A=(a_{ij})_{n\times n}$, we define the norm of $A$ as follows
$
\left\|A\right\|=\left(\sum\limits_{i,j=1}^{n}{{|a_{ij}|}^2}\right)^{1/2}.
$
\\
It is known that $\left|Ax\right|\leq \left\|A\right\|\left|x\right|$ for any vector $x\in\mathbb R^n$. In particular, if $A$ is invertible,  then we have
$$
\left\|A\right\|^{-n}\leq \left|\rm det(A^{-1})\right|\leq \left\|A^{-1}\right\|^{n}.
$$
\vskip 5pt
%Let $\Phi$ be a locally integrable function on $\mathbb R^n$. Recall that  the matrix Hausdorff operator $H_{\Phi,A}$ associated to the kernel function $\Phi$ is  defined by
%$$
%H_{\Phi, A}(f)(x)=\int\limits_{\mathbb R^n}{\frac{\Phi(y)}{|y|^n}f(A(y) x)dy},\,x\in\mathbb R^n,
%$$
%where $A(y)$ is an $n\times n$ invertible matrix for almost everywhere $y$ in the support of $\Phi$.
Our first main result in this section is as follows.
\begin{theorem}\label{MatrixHardy-Hardy}
Let $0<p\leq 1$, $1<q<\infty$, $1\leq \rho_A<\infty$, $\alpha\in [n(1-1/q),\infty)$, and $\omega_1=|x|^{\beta_1}$, $\omega_2=|x|^{\beta_2}$ with $\beta_1,\beta_2\in (-n,0]$, and there exist $m, M\in \mathbb Z$ such that  $2^m< \|A^{-1}(y)\|\leq 2^M$, a.e $y\in\rm supp(\Phi)$. Suppose that the following conditions are true:
\begin{align}\label{cond_matrix}
\|A^{-1}(y)\|.\|A(y)\|\leq \rho_A,\,\textit{ for all }\, y\in\,\textit{\rm supp}(\Phi),
\end{align}\label{Phi_matrix}
\begin{align}
\int_{\mathbb R^n}\frac{|\Phi(y)|}{|y|^n}dy<\infty.
\end{align}
Then, we have
$
\|H_{\Phi,A}(f)\|_{H{\mathop{K}\limits^{.}}^{\alpha,p}_{q}(\omega_1,\omega_2)}\lesssim \Big(\int_{\mathbb R^n}\frac{|\Phi(y)|}{|y|^n}dy\Big).\|f\|_{H{\mathop{K}\limits^{.}}^{\alpha,p}_{q}(\omega_1,\omega_2)},\,\,\textit{for all}\, f\in H{\mathop{K}\limits^{.}}^{\alpha,p}_{q}(\omega_1,\omega_2).
$
\end{theorem}
%%%%%%%%%%%%%%%%%%%%%
%%%%%%%%%%%%%%%%%%%%%
\begin{proof}
Similarly to the proof of Theorem \ref{Hardy-Hardy}, it suffices to prove that 
\begin{align}\label{matrix_uniform}
\|H_{\Phi,A}(a)\|_{H{\mathop{K}\limits^{.}}^{\alpha,p}_{q}(\omega_1,\omega_2)}\lesssim \int_{\mathbb R^n}\frac{|\Phi(y)|}{|y|^n}dy,
\end{align}
with non negative integer $s\geq [\alpha-n(1-1/q)]$ and for any central $(\alpha, q, s;\omega_1, \omega_2)_0$-atom $a$  with ${\rm supp}(a)\subset B_{j_a}$ and $\|a\|_{L^q(\omega_2)}\lesssim \omega_1(B_{j_a})^{\frac{-\alpha}{n}}$.
\vskip 5pt
Now, we write
$
H_{\Phi, A}(a)(x)= \sum\limits_{k=m+1}^M\,\int_{2^{k-1}<\|A^{-1}(y)\|\leq 2^k} \dfrac{\Phi(y)}{|y|^n }a(A(y)x)dy:= \sum\limits_{k=m+1}^M c_k(x).
$
\\
Note that ${\rm supp}(a)\subset B_{j_a}$. Let $x\in\mathbb R^n$ such that $|x|> 2^{j_a+k}$. By  $\|A^{-1}(y)\|\in (2^{k-1},2^k]$, we have
$$
|A(y)x|\geq \|A^{-1}(y)\|^{-1}|x|> 2^{-k}.2^{j_a+k}=2^{j_a}.
$$
This implies $c_k(x)=0$. From this, we immediately have
\begin{align}\label{matrix_suppck}
{\rm supp}(c_k)\subset B_{j_a+k}.
\end{align}
By using the Minkowski inequality, it gives
\begin{align}
\|c_k\|_{L^q(\omega_2)}&=\Big(\int_{\mathbb R^n}\Big|\,\int_{2^{k-1}<\|A^{-1}(y)\|\leq 2^k}\dfrac{\Phi(y)}{|y|^n}a(A(y)x)dy\Big|^q\omega_2(x)dx\Big)^{\frac{1}{q}}\nonumber
\\
&\leq \int_{2^{k-1}<\|A^{-1}(y)\|\leq 2^k}\frac{|\Phi(y)|}{|y|^n}\Big(\int_{\mathbb R^n} |a(A(y)x)|^q\omega_2(x)dx\Big)^{\frac{1}{q}}dy.\nonumber
\end{align}
From formula change of variables, $\beta_2\leq 0$ and $\|a\|_{L^q(\omega_2)}\lesssim\omega_1(B_{j_a})^{\frac{-\alpha}{n}}$, we get
\begin{align}%\label{matrix_esaLq}
&\Big(\int_{\mathbb R^n} |a(A(y)x)|^q\omega_2(x)dx\Big)^{\frac{1}{q}}=\Big(\int_{\mathbb R^n}|a(z)|^q|A^{-1}(y)z|^{\beta_2}|{\rm det}A^{-1}(y)|dz\Big)^{\frac{1}{q}}\nonumber
\\
&\leq \|A(y)\|^{\frac{-\beta_2}{q}}|{\rm det}A^{-1}(y)|^{\frac{1}{q}}\Big(\int_{\mathbb R^n}|a(z)|^q|z|^{\beta_2}dz\Big)^{\frac{1}{q}}\lesssim  \|A(y)\|^{\frac{-\beta_2}{q}}|{\rm det}A^{-1}(y)|^{\frac{1}{q}}\omega_1(B_{j_a})^{\frac{-\alpha}{n}}.\nonumber
\end{align}
This leads to
\begin{align}\label{matrix_ckLq}
&\|c_k\|_{L^q(\omega_2)}\lesssim \Big(\int_{2^{k-1}<\|A^{-1}(y)\|\leq 2^k}\frac{|\Phi(y)|}{|y|^n}\|A(y)\|^{\frac{-\beta_2}{q}}|{\rm det}A^{-1}(y)|^{\frac{1}{q}}dy\Big)\omega_1(B_{j_a+k})^{\frac{-\alpha}{n}}\Big(\frac{\omega_1(B_{j_a+k})}{\omega_1(B_{j_a})}\Big)^{\frac{\alpha}{n}}.
\end{align}
Consequently, by letting $\|A^{-1}(y)\|\in (2^{k-1},2^k]$ and having the inequality (\ref{omega1power}) above, we have
\begin{align}\label{matrix_normck}
\|c_k\|_{L^q(\omega_2)}&\lesssim  \Big(\int_{2^{k-1}<\|A^{-1}(y)\|\leq 2^k}\frac{|\Phi(y)|}{|y|^n}\|A(y)\|^{\frac{-\beta_2}{q}}|{\rm det}A^{-1}(y)|^{\frac{1}{q}}\|A^{-1}(y)\|^{\alpha+\frac{\beta_1\alpha}{n}}dy\Big)\omega_1(B_{j_a+k})^{\frac{-\alpha}{n}}
\nonumber
\\
&:=\theta_k.\omega_1(B_{j_a+k})^{\frac{-\alpha}{n}}.
\end{align}
For any $|\zeta|\leq s$, by  $c_k.x^{\zeta}\in L^1(\mathbb R^n)$, we get
\begin{align}
\int_{\mathbb R^n}c_k(x)x^{\zeta}dx&=\int_{\mathbb R^n}\,\Big(\int_{2^{k-1}<\|A^{-1}(y)\|\leq 2^k}\dfrac{\Phi(y)}{|y|^n}a(A(y)x)dy\Big)x^{\zeta}dx=\int_{2^{k-1}<\|A^{-1}(y)\|\leq 2^k}\dfrac{\Phi(y)}{|y|^n}\Big(\int_{\mathbb R^n}a(A(y)x)x^{\zeta}dx\Big)dy.\nonumber
\end{align}
Remark that by $\int_{\mathbb R^n}a(x)x^{\gamma}dx=0$ for all $|\gamma|\leq s$,  we deduce that
\begin{align}
\int_{\mathbb R^n}a(A(y)x)x^{\zeta}dx &=\int_{\mathbb R^n}a(z).(A^{-1}(y)z)^{\zeta}.|{\rm det} A^{-1}(y)|dz=|{\rm det} A^{-1}(y)|\int_{\mathbb R^n}a(z).(\sum\limits_{|\gamma|\leq |\zeta|}d_{A^{-1}(y),\gamma}.z^{\gamma})dz =0.\nonumber
\end{align}
Hence, one has
\begin{align}\label{matrix_inck=0}
\int_{\mathbb R^n}a(A(y)x)x^{\zeta}dx=0,\,\textit{\rm for all}\, |\zeta|\leq s.
\end{align}
Also, note that there exists $r_a\in\mathbb Z$ satisfying $a=0$ on $B_{r_a}$. By letting $x\in\mathbb R^n$ such that $|x|\leq \frac{r_a.2^{k-1}}{\rho_A}$ and assuming (\ref{cond_matrix}) and $\|A^{-1}(y)\|\in (2^{k-1},2^k]$, we have
$$
|A(y)x|\leq \|A(y)\|.|x|\leq \rho_A\|A^{-1}(y)\|^{-1}.|x|<\rho_A.2^{-k+1}.\frac{r_a.2^{k-1}}{\rho_A}=r_a,
$$
which implies $c_k(x)=0$. Thus,
\begin{align}\label{matrix_neighck}
c_k=0\,\textit{ \rm on }\, B_{\frac{r_a.2^{k-1}}{\rho_A}}.
\end{align}
Denote
$
c_{j_a,k} = \left\{ \begin{array}{l}
\dfrac{c_k}{\theta_k},\textit{ \rm if }\, \theta_k\neq 0,
\\
0,\,\,\,\textit\,{\rm otherwise}.\nonumber
\end{array} \right.
$
\\
We have
$$
H_{\Phi, A}(a)(x)= \sum\limits_{k=m+1}^M \theta_k. c_{j_a,k}(x).
$$
Note that, by (\ref{matrix_suppck}), (\ref{matrix_normck}), (\ref{matrix_inck=0}) and (\ref{matrix_neighck}), it is clear to see that $c_{j_a,k}$ is also central $(\alpha, q,s; \omega_1,\omega_2)_0$-atom. Hence $
H_{\Phi,A}(a)\in {\mathop{F}\limits^.}^{\alpha, q, s}_p(\omega_1,\omega_2)$. From these, by applying Theorem \ref{equiv} and (\ref{cond_matrix}), and having $|{\rm det} A^{-1}(y)|\leq \|A^{-1}(y)\|^n$, we estimate
\begin{align}
&\|H_{\Phi,A}(a)\|_{H{\mathop{K}\limits^.}^{\alpha, p}_q(\omega_1,\omega_2)}
\lesssim \|H_{\Phi, A}(a)\|_{{\mathop{F}\limits^.}^{\alpha, q, s}_p(\omega_1,\omega_2)}\leq \Big(\sum\limits_{k=m+1}^M |\theta_k|^p\Big)^{\frac{1}{p}}\lesssim \sum\limits_{k=m+1}^M |\theta_k|\nonumber
\\
&=\sum\limits_{k=m+1}^M\,\int_{2^{k-1}<\|A^{-1}(y)\|\leq 2^k}\frac{|\Phi(y)|}{|y|^n}\|A(y)\|^{\frac{-\beta_2}{q}}|{\rm det}A^{-1}(y)|^{\frac{1}{q}}\|A^{-1}(y)\|^{\alpha+\frac{\beta_1\alpha}{n}}dy\nonumber
\\
&\lesssim \sum\limits_{k=m+1}^M \,\int_{2^{k-1}<\|A^{-1}(y)\|\leq 2^k}\frac{|\Phi(y)|}{|y|^n}dy=\int_{\mathbb R^n}\dfrac{|\Phi(y)|}{|y|^n}dy.\nonumber
\end{align}
 This implies that the inequality (\ref{matrix_uniform}) is valid. 
Therefore, the proof of the theorem is completed.
\end{proof}
%%%%%%%%%%%%%%%%%%
%%%%%%%%%%%%%%%%%%%%
\begin{theorem}\label{matrix_Hardy-Herz}
Let $1<q<\infty, \alpha\in [n(1-1/q),\infty)$, $\omega_1=|x|^{\beta_1}$, $\omega_2=|x|^{\beta_2}$ with $\beta_1,\beta_2\in (-n,0]$.
\\
{\rm (i)} If $p=1$ and 
$$
\mathcal C_7= \int_{\mathbb R^n}\frac{|\Phi(y)|}{|y|^n}\|A(y)\|^{\frac{-\beta_2}{q}}|{\rm det}A^{-1}(y)|^{\frac{1}{q}}\|A^{-1}(y)\|^{\alpha+\frac{\beta_1\alpha}{n}}dy <\infty,
$$
then we have
$
\|H_{\Phi, A}(f)\|_{{\mathop{K}\limits^{.}}^{\alpha,1}_{q}(\omega_1,\omega_2)}\lesssim \mathcal C_7.\|f\|_{H{\mathop{K}\limits^{.}}^{\alpha, 1}_{q}(\omega_1,\omega_2)},\,\,\textit{for all}\, f\in H{\mathop{K}\limits^{.}}^{\alpha, 1}_{q}(\omega_1,\omega_2).
$
\\
{\rm (ii)} If $0<p<1$, $\sigma>\frac{1-p}{p}$, and
\begin{align}
\mathcal C_8 &= \int_{\mathbb R^n}\frac{|\Phi(y)|}{|y|^n}\|A(y)\|^{\frac{-\beta_2}{q}}|{\rm det}A^{-1}(y)|^{\frac{1}{q}}\|A^{-1}(y)\|^{\alpha+\frac{\beta_1\alpha}{n}}
\times\nonumber
\\
&\times\Big(\chi_{\{\|A^{-1}(y)\|\leq 1\}}(y).|{\rm log}_2\|A^{-1}(y)\||^{\sigma} + \chi_{\{\|A^{-1}(y)\|> 1\}}(y).\big({\rm log}_2\|A^{-1}(y)\| +1 \big)^{\sigma}\Big)dy<\infty,\nonumber
\end{align}
then we have
$
\|H_{\Phi, A}(f)\|_{{\mathop{K}\limits^{.}}^{\alpha,p}_{q}(\omega_1,\omega_2)}\lesssim C_8.\|f\|_{H{\mathop{K}\limits^{.}}^{\alpha, p}_{q}(\omega_1,\omega_2)},\,\,\textit{ for all}\, f\in H{\mathop{K}\limits^{.}}^{\alpha, p}_{q}(\omega_1,\omega_2).
$
\end{theorem}
%%%%%%%%%%%%%%%%%%%%%
\begin{proof}
Let us fix non negative integer $s\geq [\alpha-n(1-1/q)]$ and $a$ be any central $(\alpha, q, s;\omega_1, \omega_2)_0$-atom with ${\rm supp}(a)\subset B_{j_a}$ and $\|a\|_{L^q(\omega_2)}\lesssim \omega_1(B_{j_a})^{\frac{-\alpha}{n}}$.
To complete the proof of the theorem, it suffices to prove that the following inequality holds
\begin{align}\label{matrix_uniform1}
\|H_{\Phi, A}(a)\|_{{\mathop{K}\limits^{.}}^{\alpha,p}_{q}(\omega_1,\omega_2)}\lesssim \left\{ \begin{array}{l}
\mathcal C_7,\,\,p=1,
\\
\mathcal C_8,\,\,p\in (0,1)\,\textit{\rm and}\,\,\sigma>\frac{(1-p)}{p}.
\end{array} \right.
\end{align}
Now, we will decompose as follows
\begin{align}\label{xichmack}
|H_{\Phi,A}(a)(x)|&\leq\sum\limits_{k\in\mathbb Z}\,\int_{2^{k-1}<\|A^{-1}(y)\|\leq 2^k} \dfrac{|\Phi(y)|}{|y|^n}|a(A(y)x)|dy:= \sum\limits_{k\in\mathbb Z} \widetilde{c_k}(x).
\end{align}
From (\ref{matrix_suppck}) and (\ref{matrix_normck}), we instantly have
\begin{align}\label{matrix_suppnormck}
{\rm supp}(\widetilde{c}_k)\subset B_{j_a+k}\,\textit{ \rm and }\,\|\widetilde{c}_k\|_{L^q(\omega_2)}\lesssim \theta_k.\omega_1(B_{j_a+k})^{\frac{-\alpha}{n}}.
\end{align}
By setting up $\widetilde{c}_{j_a,k}$ as follows
$
\widetilde{c}_{j_a,k} = \left\{ \begin{array}{l}
\dfrac{\widetilde{c}_k}{\theta_k},\textit{\rm if }\, \theta_k\neq 0,
\\
0,\,\,\textit\,{\rm otherwise},\nonumber
\end{array} \right.
$
\\
we have 
$\sum\limits_{k\in\mathbb Z} \widetilde{c}_k= \sum\limits_{k\in\mathbb Z}\theta_k.\widetilde{c}_{j_a,k},
$ 
where 
$
\theta_k=\int_{2^{k-1}<\|A^{-1}(y)\|\leq 2^k}\frac{|\Phi(y)|}{|y|^n}\|A(y)\|^{\frac{-\beta_2}{q}}|{\rm det}A^{-1}(y)|^{\frac{1}{q}}\|A^{-1}(y)\|^{\alpha+\frac{\beta_1\alpha}{n}}dy.
$
\\
From (\ref{matrix_suppnormck}), it is easy to see that $\widetilde{c}_{j_a,k}$ is a dyadic central $(\alpha, q,\omega_1,\omega_2)$-unit. As a consequence, by Theorem \ref{blockHerz}, we lead to
\begin{align}
\|H_{\Phi, A}(a)\|_{{\mathop{K}\limits^{.}}^{\alpha,p}_{q}(\omega_1,\omega_2)}\leq \|\sum\limits_{k\in\mathbb Z} \widetilde{c_k}\|_{{\mathop{K}\limits^{.}}^{\alpha,p}_{q}(\omega_1,\omega_2)}\lesssim \Big(\sum\limits_{k\in\mathbb Z} |\theta_k|^p\Big)^{\frac{1}{p}}.\nonumber
\end{align}
Case 1: $p=1$. It follows that
\begin{align}
\sum\limits_{k\in\mathbb Z} |\theta_k| &=\sum\limits_{k\in\mathbb Z}\,\int_{2^{k-1}<\|A^{-1}(y)\|\leq 2^k}\frac{|\Phi(y)|}{|y|^n}\|A(y)\|^{\frac{-\beta_2}{q}}|{\rm det}A^{-1}(y)|^{\frac{1}{q}}\|A^{-1}(y)\|^{\alpha+\frac{\beta_1\alpha}{n}}dy=\mathcal C_7.\nonumber
\end{align}
Case 2: $p\in (0,1)$ and $\sigma>\frac{1-p}{p}$. By the H\"{o}lder inequality again, we also have
\begin{align}
&\Big(\sum\limits_{k\in\mathbb Z} |\theta_k|^p\Big)^{\frac{1}{p}}\lesssim \sum\limits_{k\in\mathbb Z} |k|^{\sigma}|\theta_k|=\sum\limits_{k=1}^{\infty}\,\int_{2^{k-1}<\|A^{-1}(y)\|\leq 2^k}\frac{|\Phi(y)|}{|y|^n}\|A(y)\|^{\frac{-\beta_2}{q}}|{\rm det}A^{-1}(y)|^{\frac{1}{q}}\|A^{-1}(y)\|^{\alpha+\frac{\beta_1\alpha}{n}}|k|^{\sigma}dy + \nonumber
\\
&+\sum\limits_{k=-\infty}^0\,\int_{2^{k-1}<\|A^{-1}(y)\|\leq 2^k}\frac{|\Phi(y)|}{|y|^n}\|A(y)\|^{\frac{-\beta_2}{q}}|{\rm det}A^{-1}(y)|^{\frac{1}{q}}\|A^{-1}(y)\|^{\alpha+\frac{\beta_1\alpha}{n}}|k|^{\sigma}dy\nonumber
\\
&\lesssim \sum\limits_{k=1}^{\infty} \,\int_{2^{k-1}<\|A^{-1}(y)\|\leq 2^k}\frac{|\Phi(y)|}{|y|^n}\|A(y)\|^{\frac{-\beta_2}{q}}|{\rm det}A^{-1}(y)|^{\frac{1}{q}}\|A^{-1}(y)\|^{\alpha+\frac{\beta_1\alpha}{n}}({\rm log}_2\|A^{-1}(y)\|+1)^{\sigma}dy\,+\nonumber
\\
&+\sum\limits_{k=-\infty}^0 \,\int_{2^{k-1}<\|A^{-1}(y)\|\leq 2^k}\frac{|\Phi(y)|}{|y|^n}\|A(y)\|^{\frac{-\beta_2}{q}}|{\rm det}A^{-1}(y)|^{\frac{1}{q}}\|A^{-1}(y)\|^{\alpha+\frac{\beta_1\alpha}{n}}|{\rm log}_2\|A^{-1}(y)\||^{\sigma}dy= \mathcal C_8\nonumber.
\end{align}
Thus, the inequality (\ref{matrix_uniform1}) is achieved. Therefore, the theorem is completely proved.
\end{proof}
%%%%%%%%%%%%%%%%%%%%%
%%%%%%%%%%%%%%%%%%%%%%%%
\begin{theorem}\label{matrix_Hardy-Herz1}
Suppose $1<q<\infty, \alpha\in [n(1-1/q),\infty)$, $\omega_2=|x|^{\beta_2}$ with $\beta_2\in (-n,0]$ and $\omega_1\in A_1$ with the finite critical index $r_{\omega_1}$ for the reverse H\"{o}lder condition and $\delta\in (1,r_{\omega_1})$.
\\
{\rm(i)} If $p=1$ and 
\begin{align}
\mathcal C_9 &= \int_{\mathbb R^n}\frac{|\Phi(y)|}{|y|^n}\|A(y)\|^{\frac{-\beta_2}{q}}|{\rm det}A^{-1}(y)|^{\frac{1}{q}}\times\nonumber
\\
&\,\,\,\times\Big(\chi_{\{\|A^{-1}(y)\|\leq 1\}}(y).\|A^{-1}(y)\|^{\frac{\alpha(\delta-1)}{\delta}}+ \chi_{\{\|A^{-1}(y)\|>1\}}(y).\|A^{-1}(y)\|^{\alpha}\Big)dy<\infty,\nonumber
\end{align}
then we have
$
\|H_{\Phi, A}(f)\|_{{\mathop{K}\limits^{.}}^{\alpha,1}_{q}(\omega_1,\omega_2)}\lesssim\mathcal C_9.\|f\|_{H{\mathop{K}\limits^{.}}^{\alpha, 1}_{q}(\omega_1,\omega_2)},\,\,\textit{for all}\, f\in H{\mathop{K}\limits^{.}}^{\alpha, 1}_{q}(\omega_1,\omega_2).
$
\\
 {\rm(ii)} If $0<p<1$, $\sigma>\frac{1-p}{p}$ and
\begin{align}
\mathcal C_{10} &=\int_{\mathbb R^n}\frac{|\Phi(y)|}{|y|^n}\|A(y)\|^{\frac{-\beta_2}{q}}|{\rm det}A^{-1}(y)|^{\frac{1}{q}}\Big(\chi_{\{\|A^{-1}(y)\|\leq 1\}}(y).\|A^{-1}(y)\|^{\frac{\alpha(\delta-1)}{\delta}}|{\rm log}_2\|A^{-1}(y)\||^{\sigma} +\nonumber
\\
&\,\,\,\,\,\,\,\,\,+ \chi_{\{\|A^{-1}(y)\|>1\}}(y).\|A^{-1}(y)\|^{\alpha}\big({\rm log}_2\|A^{-1}(y)\| +1 \big)^{\sigma}\Big)dt<\infty,\nonumber
\end{align}
then we have
$
\|H_{\Phi, A}(f)\|_{{\mathop{K}\limits^{.}}^{\alpha,p}_{q}(\omega_1,\omega_2)}\lesssim\mathcal C_{10}.\|f\|_{H{\mathop{K}\limits^{.}}^{\alpha, p}_{q}(\omega_1,\omega_2)},\,\,\textit{for all}\, f\in H{\mathop{K}\limits^{.}}^{\alpha, p}_{q}(\omega_1,\omega_2).
$
\end{theorem}
%%%%%%%%%%%%%%%%%%%
\begin{proof}
To prove this theorem, it suffices to obtain the following estimation  
\begin{align}
\|H_{\Phi, A}(a)\|_{{\mathop{K}\limits^{.}}^{\alpha,p}_{q}(\omega_1,\omega_2)}\lesssim \left\{ \begin{array}{l}
\mathcal C_9,\,\,p=1,
\\
\mathcal C_{10},\,p\in (0,1)\,\textit{\rm and}\,\,\sigma>\frac{(1-p)}{p},
\end{array} \right.\nonumber
\end{align}
where non negative integer $s\geq [\alpha-n(1-1/q)]$ and $a$ is any central $(\alpha, q, s;\omega_1, \omega_2)_0$-atom with ${\rm supp}(a)\subset B_{j_a}$.
Recall that 
$\widetilde{c}_k(x)=\int_{2^{k-1}<\|A^{-1}(y)\|\leq 2^k} \dfrac{|\Phi(y)|}{|y|^n}|a(A(y)x)|dy\,$ with ${\rm supp}(\widetilde{c_k})\subset B_{j_a+k}.$
\\
Now, by (\ref{matrix_ckLq}) and (\ref{omega1A1}), it follows that
\begin{align}
\|\widetilde{c}_k\|_{L^q(\omega_2)}&\lesssim \omega_1(B_{j_a+k})^{\frac{-\alpha}{n}}
\left\{ \begin{array}{l}
\int_{2^{k-1}<\|A^{-1}(y)\|\leq 2^k}\frac{|\Phi(y)|}{|y|^n}\|A(y)\|^{\frac{-\beta_2}{q}}|{\rm det}A^{-1}(y)|^{\frac{1}{q}}\|A^{-1}(y)\|^{\alpha}dy,\,\,\,\,\,\,\,\,\textit{\rm if}\,\,k\geq 1,
\\
\\
\int_{2^{k-1}<\|A^{-1}(y)\|\leq 2^k}\frac{|\Phi(y)|}{|y|^n}\|A(y)\|^{\frac{-\beta_2}{q}}|{\rm det}A^{-1}(y)|^{\frac{1}{q}}\|A^{-1}(y)\|^{\frac{\alpha(\delta-1)}{\delta}}dy, \textit{\rm otherwise},
\end{array} \right.
\nonumber
\\
&:= \omega_1(B_{j_a+k})^{\frac{-\alpha}{n}}.{\eta}_k.\nonumber
\end{align}
We denote $\widetilde{c}^*_{j_a,k}$ as follows
$
\widetilde{c}^*_{j_a,k} = \left\{ \begin{array}{l}
\dfrac{\widetilde{c}_k}{{\eta}_k},\,\textit{\rm if}\,\, \eta_k\neq 0,
\\
0,\,\,\,\textit\,{\rm otherwise}.\nonumber
\end{array} \right.
$
\\
Thus, we have
$
\sum\limits_{k\in\mathbb Z} \widetilde{c}_k= \sum\limits_{k\in\mathbb Z} \eta_k.\widetilde{c}^*_{j_a,k},
$ where $\widetilde{c^*}_{j_a,k}$ is a dyadic central $(\alpha, q,\omega_1,\omega_2)$-unit. Hence, by combining the inequality (\ref{xichmack}) and Theorem \ref{blockHerz}, and estimating as the final part in the proof of Theorem \ref{matrix_Hardy-Herz} above, we also obtain that
\begin{align}
\|H_{\Phi,A}(a)\|_{{\mathop{K}\limits^{.}}^{\alpha,p}_{q}(\omega_1,\omega_2)}\lesssim \Big(\sum\limits_{k\in\mathbb Z} |\eta_k|^p\Big)^{\frac{1}{p}}\lesssim \left\{ \begin{array}{l}
\mathcal C_9,\,\,p=1,
\\
\mathcal C_{10},\,p\in (0,1)\,\textit{\rm and}\,\,\sigma>\frac{(1-p)}{p},
\end{array} \right.
\nonumber
\end{align}
which shows that the proof of this theorem is finished.
\end{proof}
%%%%%%%%%%%%%%%%%%%
\begin{theorem}\label{matrix_Hardy-Herz2}
Let $1\leq q^*<q <\infty$, $0<\alpha^*<\infty$, $\alpha\in [n(1-1/q),\infty)$, $\omega_i\in A_1$ with the finite critical index $r_{\omega_i}$ for the reverse H\"{o}lder condition and $\delta_i\in (1,r_{\omega_i})$, for all i=1,2. Assume that $q>q^*r_{\omega_2}'$ and the hypothesises (\ref{qalpha}) and (\ref{omega2/omega1}) in Theorem \ref{Hardy-Herz2} hold.
\\
{\rm (i)} If $p=1$ and 
\begin{align}
\mathcal C_{11}&=\int_{\mathbb R^n}\frac{|\Phi(y)|}{|y|^n}|{\rm det} A^{-1}(y)|^{\frac{1}{q}}\|A(y)\|^{\frac{n}{q}}\Big(\chi_{\{\|A^{-1}(y)\|\leq 1\}}(y).\|A^{-1}(y)\|^{\gamma_1}+\nonumber
\\
&\,\,\,\,\,\,\,\,\,\,\,\,\,\,\,\,\,\,\,\,\,\,\,\,\,+\chi_{\{\|A^{-1}(y)\|>1\}}(y).\|A^{-1}(y)\|^{\gamma_2}\Big)dy<\infty,\nonumber
\end{align}
then we have
$
\|H_{\Phi, A}(f)\|_{{\mathop{K}\limits^{.}}^{\alpha^*,1}_{q^*}(\omega_1,\omega_2)}\lesssim \mathcal C_{11}.\|f\|_{H{\mathop{K}\limits^{.}}^{\alpha, 1}_{q}(\omega_1,\omega_2)},\,\,\textit{ for all }\, f\in H{\mathop{K}\limits^{.}}^{\alpha, 1}_{q}(\omega_1,\omega_2).
$
\\
{\rm (ii)} If $0<p<1$, $\sigma>\frac{1-p}{p}$ and
\begin{align}
\mathcal C_{12}&=\int_{\mathbb R^n}\frac{|\Phi(y)|}{|y|^n}|{\rm det} A^{-1}(y)|^{\frac{1}{q}}\|A(y)\|^{\frac{n}{q}}\Big(\chi_{\{\|A^{-1}(y)\|\leq 1\}}(y).\|A^{-1}(y)\|^{\gamma_1}|{\rm log}_2\|A^{-1}(y)\||^{\sigma}+
\nonumber
\\
&\,\,\,\,\,\,\,\,\,\,\,\,\,\,\,\,\,\,+\chi_{\{\|A^{-1}(y)\|>1\}}(y).\|A^{-1}(y)\|^{\gamma_2}\big({\rm log}_2\|A^{-1}(y)\| +1 \big)^{\sigma}\Big)dy<\infty,\nonumber
\end{align}
then we have
$
\|H_{\Phi, A}(f)\|_{{\mathop{K}\limits^{.}}^{\alpha^*,p}_{q^*}(\omega_1,\omega_2)}\lesssim \mathcal C_{12}.\|f\|_{H{\mathop{K}\limits^{.}}^{\alpha, p}_{q}(\omega_1,\omega_2)},\,\,\textit{ for all }\, f\in H{\mathop{K}\limits^{.}}^{\alpha, p}_{q}(\omega_1,\omega_2).
$
\\
Here $\gamma_1,\gamma_2$ are defined in Theorem \ref{Hardy-Herz2}.
\end{theorem}
%%%%%%%%%%%%%%%%%%%%
%%%%%%%%%%%%%%%%%%%%
\begin{proof}
Let fix nonnegative integer $s\geq [\alpha-n(1-1/q)]$. Let $a$ be any central $(\alpha, q, s;\omega_1, \omega_2)_0$-atom with ${\rm supp}(a)\subset B_{j_a}$. In order to prove the theorem, it suffices to show that
\begin{align}\label{matrix_uniform3}
\|H_{\Phi,A}(a)\|_{{\mathop{K}\limits^{.}}^{\alpha^*,p}_{q^*}(\omega_1,\omega_2)}\lesssim \left\{ \begin{array}{l}
\mathcal C_{11},\,\,p=1,
\\
\mathcal C_{12},\,\,p\in (0,1)\,\textit{\rm and}\,\sigma>\frac{(1-p)}{p}.
\end{array} \right.
\end{align}
We first recall that  
$
\widetilde{c}_k(x)=\int_{2^{k-1}<\|A^{-1}(y)\|\leq 2^k} \dfrac{|\Phi(y)|}{|y|^n}.|a(A(y)x)|dy,
$
with ${\rm supp}(\,\widetilde{c}_k)\subset B_{j_a+k}$.
\\
By the Minkowski inequality again, one has
\begin{align}
\|\widetilde{c}_k\|_{L^{q^*}(\omega_2)}&=\Big(\int_{B_{j_a+k}}\Big|\int_{2^{k-1}<\|A^{-1}(y)\|\leq 2^k} \dfrac{\Phi(y)}{|y|^n}.a(A(y)x)dy\Big|^{q^*}\omega_2(x)dx\Big)^{\frac{1}{q^*}}\nonumber
\\
&\leq \int_{2^{k-1}<\|A^{-1}(y)\|\leq 2^k}\frac{|\Phi(y)|}{|y|^n}\Big(\int_{B_{j_a+k}}|a(A(y)x)|^{q^*}\omega_2(x)dx\Big)^{\frac{1}{q^*}}dy.\nonumber
\end{align}
As the first part in the proof of Theorem \ref{Hardy-Herz2}, there exists $r\in (1,r_{\omega_2})$ with $q=q^* r'$. From this, by the H\"{o}lder inequality and the reverse H\"{o}lder condition again, we give
\begin{align}
\Big(\int_{B_{j_a+k}}|a(A(y)x)|^{q^*}\omega_2(x)dx\Big)^{\frac{1}{q^*}} &\leq \Big(\int_{B_{j_a+k}}|a(A(y)x)|^{q}dx\Big)^{\frac{1}{q}} .\Big(\int_{B_{j_a+k}}\omega_2^r(x)dx\Big)^{\frac{1}{rq^*}}\nonumber
\\
&\lesssim |{\rm det} A^{-1}(y)|^{\frac{1}{q}}.\|a\|_{L^q(\|A(y)\|B_{j_a+k})}.|B_{j_a+k}|^{\frac{-1}{q}}\omega_2(B_{j_a+k})^{\frac{1}{q^*}}.\nonumber
\end{align}
Thus,
\begin{align}
&\|\widetilde{c}_k\|_{L^{q^*}(\omega_2)}\lesssim |B_{j_a+k}|^{\frac{-1}{q}}\omega_2(B_{j_a+k})^{\frac{1}{q^*}}.\int_{2^{k-1}<\|A^{-1}(y)\|\leq 2^k}\frac{|\Phi(y)|}{|y|^n}|{\rm det} A^{-1}(y)|^{\frac{1}{q}}.\|a\|_{L^q(\|A(y)\|B_{j_a+k})}dy.\nonumber
\end{align}
In addition, by Proposition \ref{pro2.4DFan} and the inequality $\|a\|_{L^q(\omega_2)}\lesssim \omega_1(B_{j_a})^{\frac{-\alpha}{n}}$, we have
\begin{align}
\|a\|_{L^q(\|A(y)\|B_{j_a+k})}&\lesssim |\,\|A(y)\| B_{j_a+k}|^{\frac{1}{q}}\omega_2(\|A(y)\|B_{j_a+k})^{\frac{-1}{q}}\|a\|_{L^q(\omega_2, \|A(y)\|B_{j_a+k})}
\nonumber
\\
&\lesssim |\,\|A(y)\|B_{j_a+k}|^{\frac{1}{q}}\omega_2(\|A(y)\|B_{j_a+k})^{\frac{-1}{q}}\omega_1(B_{j_a})^{\frac{-\alpha}{n}}.\nonumber
\end{align}
Consequently, by $\Big(\frac{|\,\|A(y)\|B_{j_a+k}|}{|B_{j_a+k}|}\Big)^{\frac{1}{q}}\simeq \|A(y)\|^{\frac{n}{q}}$, it implies that
\begin{align}
\|\widetilde{c}_k\|_{L^{q*}(\omega_2)} &\lesssim \int_{2^{k-1}<\|A^{-1}(y)\|\leq 2^k}\frac{|\Phi(y)|}{|y|^n}|{\rm det} A^{-1}(y)|^{\frac{1}{q}}\|A(y)\|^{\frac{n}{q}}\omega_2(B_{j_a+k})^{\frac{1}{q^*}}\omega_2(\|A(y)\|B_{j_a+k})^{\frac{-1}{q}}\omega_1(B_{j_a})^{\frac{-\alpha}{n}}dy.\nonumber
\end{align}
Because of having $\| A(y)\|\geq \|A^{-1}(y)\|^{-1}$, by letting $\|A^{-1}(y)\|\in (2^{k-1}, 2^k]$, we infer 
$$
\omega_2(\|A(y)\|B_{j_a+k})^{\frac{-1}{q}}\leq \omega_2(B_{j_a})^{\frac{-1}{q}}.
$$
Thus,
\begin{align}
\|\widetilde{c}_k\|_{L^{q*}(\omega_2)}\lesssim \omega_1(B_{j_a+k})^{\frac{-\alpha^*}{n}}.\int_{2^{k-1}<\|A^{-1}(y)\|\leq 2^k}\frac{|\Phi(y)|}{|y|^n}|{\rm det} A^{-1}(y)|^{\frac{1}{q}}\|A(y)\|^{\frac{n}{q}}.{\mathcal U}_{\omega_1,\omega_2,k,j_a}dy.\nonumber
\end{align}
Here, ${\mathcal U}_{\omega_1,\omega_2,k,j_a}$ is defined as in Theorem \ref{Hardy-Herz2}.
By the inequality  (\ref{Uomega1omega2}), we have
\begin{align}
\|\widetilde{c}_k\|_{L^{q^*}(\omega_2)}&\lesssim \omega_1(B_{j_a+k})^{\frac{-\alpha^*}{n}}.
\left\{ \begin{array}{l}
\int_{2^{k-1}<\|A^{-1}(y)\|\leq 2^k}\frac{|\Phi(y)|}{|y|^n}|{\rm det} A^{-1}(y)|^{\frac{1}{q}}\|A(y)\|^{\frac{n}{q}}\|A^{-1}(y)\|^{\gamma_2}dy,
\\
\textit{\rm if}\,\,k\geq 1,
\\
\int_{2^{k-1}<\|A^{-1}(y)\|\leq 2^k}\frac{|\Phi(y)|}{|y|^n}|{\rm det} A^{-1}(y)|^{\frac{1}{q}}\|A(y)\|^{\frac{n}{q}}\|A^{-1}(y)\|^{\gamma_1}dy,
\\
\textit{\rm otherwise},
\end{array} \right.
\nonumber
\\
&:= \omega_1(B_{j_a+k})^{\frac{-\alpha^*}{n}}.{\eta^*}_k.\nonumber
\end{align}
Now, we define $\widetilde{c}^{**}_{j_a,k}$ as follows
$
\widetilde{c}^{**}_{j_a,k} = \left\{ \begin{array}{l}
\dfrac{\widetilde{c}_k}{{\eta^*}_k},\,\textit{\rm if }\, \eta^*_k\neq 0,
\\
0,\,\,\,\,\,\,\textit{\rm otherwise}.\nonumber
\end{array} \right.
$
\\
It is obvious that
$
\sum\limits_{k\in\mathbb Z} \widetilde{c}_k= \sum\limits_{k\in\mathbb Z}\eta^*_k.\widetilde{c}^{**}_{j_a,k},
$
where $\widetilde{c}^{**}_{j_a,k}$ is a dyadic central $(\alpha^*, q^*,\omega_1,\omega_2)$-unit. Thus, by the inequality (\ref{xichmack}) and Theorem \ref{blockHerz} again, we immediately obtain
$
\|H_{\Phi,A}(a)\|_{{\mathop{K}\limits^{.}}^{\alpha^*,p}_{q^*}(\omega_1,\omega_2)}\lesssim\Big(\sum\limits_{k\in\mathbb Z} |\eta^*_k|^p\Big)^{\frac{1}{p}}.
$ By the similar proof as the final part of Theorem \ref{matrix_Hardy-Herz} above, we also obtain that 
\begin{align}
&\Big(\sum\limits_{k\in\mathbb Z} |\eta^*_k|^p\Big)^{\frac{1}{p}}\lesssim \left\{ \begin{array}{l}
\mathcal C_{11},\,\,p=1,
\\
\mathcal C_{12},\,\,p\in (0,1)\,\textit{\rm and}\,\,\sigma>\frac{(1-p)}{p}.
\end{array} \right.\nonumber
\end{align}
Consequently, the inequality (\ref{matrix_uniform3}) is valid, and hence, the proof of the theorem is completed.
\end{proof}

\end{document}